\documentclass{amsart}

\usepackage[all]{xy}
\usepackage{amsmath}
\usepackage{amssymb}
\usepackage{amsthm}
\usepackage{enumerate}

\usepackage[pdftex]{graphicx}

\theoremstyle{definition}
\newtheorem{definition}{Definition}[section]
\theoremstyle{plain}
\newtheorem{lemma}[definition]{Lemma}
\newtheorem{theorem}[definition]{Theorem}

\newtheorem{proposition}[definition]{Proposition}
\newtheorem{corollary}[definition]{Corollary}
\theoremstyle{remark}

\newtheorem{example}[definition]{Example}

\makeatletter
\@namedef{subjclassname@2020}{
  \textup{2020} Mathematics Subject Classification}
\makeatother

\newcommand{\mycl}{\operatorname{\mbox{\rm cl}}}

\newcommand{\mylcp}{\operatorname{\mbox{\bf p.an}}}
\newcommand{\myval}{\operatorname{\mbox{\bf val}}}

\begin{document}
\title[Quasi-quadratic modules in pseudo-valuation domain]{Quasi-quadratic modules in pseudo-valuation domain}
\author[M. Fujita]{Masato Fujita}
\address{Department of Liberal Arts,
Japan Coast Guard Academy,
5-1 Wakaba-cho, Kure, Hiroshima 737-8512, Japan}
\email{fujita.masato.p34@kyoto-u.jp}

\author[M. Kageyama]{Masaru Kageyama}
\address{Department of Architectural Engineering,
Faculty of Engineering,
Hiroshima Institute of Technology,
2-1-1 Miyake, Saeki-ku, Hiroshima 731-5193, Japan}
\email{m.kageyama.d4@cc.it-hiroshima.ac.jp}

\begin{abstract}
We study quasi-quadratic modules in a pseudo-valuation domain $A$ whose strict units admit a square root.
	Let $\mathfrak X_R^N$ denote the set of quasi-quadratic modules in an $R$-module $N$, where $R$ is a commutative ring.
	It is known that there exists a unique overring $B$ of $A$ such that $B$ is a valuation ring with the valuation group $(G,\leq)$ and the maximal ideal of $B$ coincides with that of $A$.
	Let $F$ be the residue field of $B$.
	In the above setting, we found a one-to-one correspondence 
	between $\mathfrak X_A^A$ and a subset of $\prod_{g \in G,g \geq e} \mathfrak X_{F_0}^F$.
\end{abstract}



\maketitle

\section{Introduction}\label{sec:intro}
Quadratic modules in the ring of univariate formal power series $E[\![X]\!]$ in the indeterminate $X$ 
were completely classified in \cite{AK} when $E$ is a euclidean field.
This result is due to simple form of elements in the ring $E[\![X]\!]$ and due to the fact that a quadratic module is finitely generated.
In \cite{FK}, the authors considerably generalized it to the case in which the ring is an overring of a valuation ring whose strict units admit a square root.
They first introduce the notion of quasi-quadratic modules which is a slight generalization of quadratic modules, and they gave a complete classification of quasi-quadratic modules of the ring under the assumption that the quasi-quadratic modules in the residue class field is already given.
A pseudo-angular component map was the key tool used in the classification.

This paper tackles the same problem, that is, classification of quasi-quadratic modules, but we consider pseudo-valuation domains defined in \cite{He} rather than overrings of a valuation ring. Let us recall the definition of a pseudo-valuation domain.

\begin{definition}\label{def:pseudo-val-domain}
	Let $A$ be a domain with the quotient field $Q(A)=K$.
	A prime ideal $\mathfrak{p}$ of $A$ is called {\it strongly prime} if $x,y\in K$ and $xy\in\mathfrak{p}$ imply that
	$x\in \mathfrak{p}$ or $y\in \mathfrak{p}$. A domain $A$ is called a {\it pseudo-valuation domain} if every prime ideal 
	of $A$ is strongly prime.
\end{definition}

It immediately follows from the definition that 
every valuation domain is a pseudo-valuation domain and
a pseudo-valuation domain is a local ring. See \cite[Proposition 1.1 and Corollary 1.3]{He}.
We fix a pseudo-valuation domain $A$ throughout the paper.
We use the following notations in this paper.
\medskip

\begin{center}
	\begin{tabular}{|c|l|}
		\hline
		Notation & Description\\
		\hline
		\hline
		$\mathfrak m$ & the maximal ideal of $A$.\\
		\hline
		$F_0$ & the residue field of $A$, i.e. $A/\mathfrak m$.\\
		\hline
		$\pi_0$ & the residue map $A \to F_0$.\\
		\hline
		$K$ & the quotient field of the domain $A$.\\
		\hline
	\end{tabular}
\end{center}
\medskip

We give a fundamental example of a pseudo-valuation domain which is not a valuation domain.

\begin{example}\label{fund-ex}
	Let us consider the local domain 
	$$A=\left\{\sum_{i=0}^\infty a_ix^i \in \mathbb C[\![X]\!] \;\middle\vert\;  a_0 \in \mathbb R\right\}.$$
	with the maximal ideal $\mathfrak{m}=X\mathbb C[\![X]\!]$. Set $B=\mathbb C[\![X]\!]$.
	The local domain $A$ is not a valuation domain of the quotient field $K=\mathbb{C}(\!(X)\!)$ but 
	a pseudo-valuation domain. In fact, $A$ is not a valuation domain of $K$ because
	$\sqrt{-1}\not\in A$ and $(\sqrt{-1})^{-1}\not\in A$.
	The remaining task is to show that every prime ideal of $A$ is strongly prime.
	By using \cite[Theorem 1.4]{He}, we have only to show that $\mathfrak{m}$ is strongly prime.
	Take elements $x, y\in K$ with $xy\in \mathfrak{m}$. 
	Note that $B$ is a valuation domain of $K$ with the maximal ideal $\mathfrak{m}$.
	Since a valuation domain becomes a pseudo-valuation domain, $B$ is a pseudo-valuation domain of $K$.
	We have $x\in\mathfrak{m}$ or $y\in\mathfrak{m}$ because $\mathfrak{m}$ is strongly prime as an ideal of $B$.
	Thus it follows that $\mathfrak{m}$ is strongly prime as an ideal of $A$.
\end{example}

We employ a technical assumption which is similar to the assumption employed in \cite{FK}. 
\begin{definition}\label{def:strict_unitA}
	An element $x$ in $A$ is a \textit{strict unit} if $\pi_0(x)=1$.
	The element $x$ \textit{admits a square root} if there exists $u \in A$ such that $x=u^2$.
\end{definition}
We assume that any strict units in $A$ admit a square root.

It is known that there exists a unique overring $B$ of $A$ which is a valuation ring in $K$ such that
the maximal ideal of $B$ coincides with $\mathfrak m$ \cite[Theorem 2.7]{He}.
We also use the following notations:
\medskip

\begin{center}
	\begin{tabular}{|c|l|}
				\hline
				Notation & Description\\
				\hline
				\hline
				$(K,\myval)$ & the valued field whose \\
				&valuation ring is $B$.\\
				\hline
				$(G,<)$ & the valuation group of $(K,\myval)$\\
				& with the identity element $e$.\\
				\hline
				$F$ & the residue class field of $(K,\myval)$.\\
				\hline
				$\pi$ & the residue class map $B \to F$\\
				& of $(K,\myval)$.\\
				\hline
				$\overline{g}$ & the equivalence class of $g \in G$\\
				& in $G/G^2$, where $G^2$ denotes\\
				& the set of squares in $G$.\\
				\hline
			\end{tabular}
\end{center}
\medskip

It is easy to demonstrate that $F_0$ is a subfield of $F$ and $\pi_0$ is the restriction of $\pi$ to $A$.
Hence we have the following commutative diagram
$$
\xymatrix@M=8pt{
	A \ar[r]^{\pi_0}  \ar@{^{(}->}[d] & F_0=A/\mathfrak{m} \ar@{^{(}->}[d] \\
	B \ar[r]_{\pi}  & F=B/\mathfrak{m} 
	\ar@{}[lu]|{\circlearrowright}
}
$$
with the natural ring homomorphisms.

We also have to slightly generalize the notion of quasi-quadratic modules.
In \cite{FK}, a quasi-quadratic module is a subset of the given ring.
We redefine a quasi-quadratic module so that it is a subset of an $A$-module.  
\begin{definition}
	Let $R$ be a commutative ring and $N$ be an $R$-module.
	A subset $M$ of $N$ is a \textit{quasi-quadratic $R$-module} in $N$ if $M+M \subseteq M$ and $a^2 M \subseteq M$ for all $a \in R$.
	Note that we always have $0 \in M$.
	A quasi-quadratic $R$-module is simply called a quasi-quadratic module when the ring $R$ is clear from the context.
	Let $\mathfrak X_R^N$ denote the set of quasi-quadratic $R$-modules in $N$.
	We simply write it by $\mathfrak X_R$ when $N=R$.
\end{definition}

We are now ready to describe the purpose of our paper.
Our goal is to construct a one-to-one correspondence between $\mathfrak X_A$ and a subset of $\prod_{g \in G, g \geq e} \mathfrak X_{F_0}^F$.
The pseudo-angular component map $\mylcp:K^\times \rightarrow F^\times$ defined in \cite{FK} is also a very useful tool in this study.
In Section \ref{sec:previous}, we recall the results in the previous study \cite{FK} including the pseudo-angular component map.
Several basic lemmas which follow immediately from them are also discussed in this section.
Section \ref{sec:qqmodule} is the main part of this paper.
We prove the above one-to-one correspondence and several properties of $\mathfrak X_A$ in this section.
As a special case, we study $\mathfrak X_A$ for the pseudo-valuation domains $A$
introduced in Example \ref{fund-ex} in Section \ref{sec:example}.
In the above sections, we assume that $F_0$ is not of characteristic two.
We consider the case in which $F_0$ is of characteristic two in Appendix \ref{sec:char2}.

Finally, we summarize notations used in this paper.
For a commutative ring $R$, let $R^\times$ denotes the set of units.
For any subset $U$ of $R$ and any element $x \in R$, we mean that  $x \in U$ and $-x \in U$ by the notation $\pm x \in U$.  
When $(S,<)$ be a linearly ordered set, for an element $c$ in $S$, the notation $S_{\geq c}$ denotes the set of elements in $S$ not smaller than $c$.

\section{Review of the previous results}\label{sec:previous}

We defined a pseudo-angular component map $\mylcp:K^{\times} \rightarrow F^{\times}$ and demonstrated the existence of a pseudo-angular component map when strict units in $B$ always admit a square root in \cite[Proposition 2.14]{FK}.
Here is the definition of a pseudo-angular component map.

\begin{definition}\label{def:pseudo}
	A \textit{pseudo-angular component map} is a map $\mylcp:K^{\times} \rightarrow F^{\times}$ satisfying the following conditions:
	\begin{enumerate}
		\item[(1)] We have $\mylcp(u)=\pi(u)$ for any $u \in B^{\times}$;
		\item[(2)] The equality $\mylcp(ux)=\pi(u) \cdot \mylcp(x)$ holds true for all $u \in B^{\times}$ and $x \in K^{\times}$;
		\item[(3)] For any $g \in G$ and nonzero $c \in F$, there exists an element $w \in K$ with $\myval(w)=g$ and $\mylcp(w)=c$;
		\item[(4)] For any nonzero elements $x_1,x_2 \in K$ with $x_1+x_2 \not=0$, we have
		\begin{itemize}
			\item $\mylcp(x_1+x_2)=\mylcp(x_1)$ if $\myval(x_1) < \myval(x_2)$;
			\item $\myval(x_1+x_2)=\myval(x_1)$ and $\mylcp(x_1+x_2)=\mylcp(x_1)+\mylcp(x_2)$ if $\myval(x_1) = \myval(x_2)$ and $\mylcp(x_1)+\mylcp(x_2) \not=0$;
		\end{itemize}
		\item[(5)] For any $x,y \in K^{\times}$ with $\overline{\myval(x)}=\overline{\myval(y)}$ and $\mylcp(x)=\mylcp(y)$, there exists $u \in K^{\times}$ such that $y=u^2x$;
		\item[(6)] For any $a,u \in K^{\times}$, there exists $k \in F^{\times}$ such that $\mylcp(au^2)=\mylcp(a)k^2$.
	\end{enumerate}
\end{definition}

We give several lemmas necessary in this paper.

\begin{lemma}\label{lem:lcp_basic1}
	Let $\mylcp:K^{\times} \rightarrow F^{\times}$ be a pseudo-angular component map.
	For any $a \in K^{\times}$, $g \in G$ and $k \in F^{\times}$, there exists $u \in K^{\times}$ such that $\myval(u)=g$ and $\mylcp(au^2)=\mylcp(a)k^2$.
\end{lemma}
\begin{proof}
	Take an element $u_0 \in K^{\times}$ such that $\myval(u_0)=g$.
	It always exists because the valuation is surjective.
	There exists $k_0 \in F^\times$ such that $\mylcp(au_0^2)=\mylcp(a)k_0^2$ by Definition \ref{def:pseudo}(6).
	Take an element $v \in B$ with $\pi(v)=k/k_0$ and set $u=u_0v$.
	It is obvious that $\myval(v)=e$ and $\myval(u)=g$.
	By Definition \ref{def:pseudo}(2), we have $\mylcp(au^2)=\mylcp(au_0^2v^2)=\mylcp(au_0^2)\pi(v)^2=\mylcp(a) \cdot k_0^2 \cdot (k/k_0)^2=\mylcp(a)k^2$.
\end{proof}

\begin{lemma}\label{lem:lcp_basic2}
	Let $\mylcp:K^{\times} \rightarrow F^{\times}$ be a pseudo-angular component map.
	Let $x,y \in K^{\times}$.
	If $\myval(x)=\myval(y)$ and $\mylcp(x)+\mylcp(y)=0$, then $\myval(x+y)>\myval(x)$. 
\end{lemma}
\begin{proof}
	It is \cite[Corollary 2.16]{FK}.
\end{proof}

\section{Quasi-quadratic modules in a pseudo-valuation domain}\label{sec:qqmodule}

We come back to the study of the pseudo-valuation domain $A$.
We use the notations given in Section \ref{sec:intro} without notice.

\begin{lemma}\label{lem:basic}
	Let $A$ be a pseudo-valuation domain.
	Let $x \in B$.
	If $\pi(x) \in F_0$, then $x \in A$.
	If, in addition, $\pi(x) \neq 0$, $x \in A^\times$.
\end{lemma}
\begin{proof}
	When $\pi(x)=0$, then $x \in \mathfrak m \subset A$.
	When $\pi(x) \neq 0$, there exists $y \in A$ with $\pi_0(y)=\pi(x)$ by the definition of $F_0$.
	Since $\pi_0(y) \neq 0$, we have $y \in A^{\times}$.
	In particular, we get $1/y \in A^\times$.
	Set $z=x/y-1$.
	We have $\pi(z)=\pi(x/y)-1=\pi(x)/\pi(y)-1=\pi(x)/\pi_0(y)-1=0$.
	It implies $z \in \mathfrak m \subset A$.
	Therefore, we get $x = (x/y) \cdot y =(1+z)y \in A$. 
\end{proof}

\begin{corollary}\label{cor:prepseudo}
	Let $A$ be a pseudo-valuation domain.
	Any strict units in $B$ admit a square root if and only if any strict units in $A$ admit a square root.
\end{corollary}
\begin{proof}
	Strict units in $B$ and their square roots belong to $A$ by Lemma \ref{lem:basic} because their images under $\pi$ is $\pm 1$ and they belong to $F_0$. 
\end{proof}

There exists a pseudo-angular component map $\mylcp:K^\times \rightarrow F^\times$ when any strict units in the valuation ring $B$ admit a square root by \cite[Proposition 2.14]{FK}.
Therefore, by the above corollary, if any strict units in the pseudo-valuation domain $A$ admit a square root, a pseudo-angular component map $\mylcp:K^\times \rightarrow F^\times$ exists.
We assume that any strict units in $A$ admit a square root,  and we fix a pseudo-angular component map $\mylcp$ throughout this section.

Let $S$ be a subset of a quasi-quadratic $A$-module $M$.
The smallest quasi-quadratic $A$-submodule of $M$ containing the set $S$ is called the \textit{quasi-quadratic closure} of $S$ in $M$ and denoted by $\mycl_A(S;M)$.
We have 
\begin{align*}
	&\mycl_A(S;\!M)=\Bigl\{\sum_{i=1}^n a_i^2s_i \Bigm\vert n \in \mathbb Z_{> 0}, a_i \in A,  s_i \in S\ (1 \leq \forall i \leq n)\Bigr\}.
\end{align*}
It is denoted by $\mycl_A(S)$ when $M=A$. 
The subscript $A$ is omitted when it is clear from the context.

We construct a one-to-one correspondence between $\mathfrak X_A$ and a subset of 
$\prod_{g \in G, g \geq e} \mathfrak X_{F_0}^F$ step by step.
We first define important notions.
\begin{definition}\label{def:psi}
	For any $g \in G_{\geq e}$ and a subset $\mathcal M$ of $A$,
	we set
\begin{align*}
			M_g(\mathcal M)=
\{\mylcp(x) \in F^\times\mid x \in \mathcal M \setminus \{0\}, \myval(x)=g\} \cup \{0\}.
\end{align*} 
	For $g \in G_{\geq e}$ and a subset $M$ of $F$, we put
	\begin{align*}
			\Psi_1(M,g)&=\{x \in A \setminus \{0\}\mid \overline{\myval(x)} = \overline{g} \text{ and }
			((\myval(x)=g \text{ and }\mylcp(x) \in M )\\
			&\quad \text{ or } (\myval(x)>g \text{ and } \mylcp(x) \in \mycl_F(M)))\},\\
			\Psi_2(M,g)&=\{x \in A \setminus \{0\}\mid \exists u \in \Psi_1(M,g) \text{ such that }\\
			& \qquad - u \in \Psi_1(M,g) \text{ and }\myval(u)<\myval(x)\} \text{ and }\\
			\Psi(M,g)&=\Psi_1(M,g) \cup \Psi_2(M,g)  \cup \{0\}.
		\end{align*}
\end{definition}

We want to show that both $M_g(\mathcal M)$ and $\Psi(M,g)$ are quasi-quadratic modules.
To this end, we prepare several lemmas.

\begin{lemma}\label{lem:basic2}
	Let $A$ be a pseudo-valuation domain such that any strict units admit a square root.
	Let $\mathcal M$ be a quasi-quadratic $A$-module in $A$.
	For any $g,h \in G$ with $g \geq e$ and $h>e$, and $0 \neq c \in M_g(\mathcal M)$, the inclusions  $F_0^2c \subseteq M_g(\mathcal M)$ and $F^2c \subseteq M_{gh^2}(\mathcal M)$ hold true.
\end{lemma}
\begin{proof}
	Fix $x \in \mathcal M$ with $\myval(x)=g$ and $\mylcp(x)=c$.
	We first show the inclusion $F_0^2c \subseteq M_g(\mathcal M)$.
	It is obvious that $0 \in M_g(\mathcal M)$.
	Fix an arbitrary element $k \in F_0^{\times}$.
	Take $u \in B^{\times}$ such that $\pi(u)=k$.
	We have $u \in A^{\times}$ by Lemma \ref{lem:basic}.
	We get $u^2x \in \mathcal M$ and $\mylcp(u^2x)=k^2c$ by Definition \ref{def:pseudo}(2).
	It implies that $k^2c \in M_g(\mathcal M)$.
	The first inclusion has been demonstrated.
	
	We next demonstrate the inclusion $F^2c \subseteq M_{gh^2}(\mathcal M)$.
	It is obvious that $0 \in M_{gh^2}(\mathcal M)$.
	Fix an arbitrary element $k \in F^{\times}$.
	By Lemma \ref{lem:lcp_basic1}, there exists $u \in K^{\times}$ with $\myval(u)=h$ and $\mylcp(xu^2)=k^2c$.
	Since $h>e$, we have $u \in \mathfrak m \subset A$.
	Therefore, we have $xu^2 \in \mathcal M$.
	We obtain $k^2c \in M_{gh^2}(\mathcal M)$ because $\myval(xu^2)=gh^2$ and $\mylcp(xu^2)=k^2c$.
\end{proof}

\begin{lemma}\label{lem:basic3}
	Let $A$ be a pseudo-valuation domain such that $2$ is a unit and any strict units admit a square root.
	Let $\mathcal M$ be a quasi-quadratic $A$-module in $A$.
	Assume that there exists $0 \neq u \in A$ with $\pm u \in \mathcal M$.
	Then, any element $x \in A$ with $\myval(x)>\myval(u)$ belongs to the quasi-quadratic module $\mathcal M$. 
\end{lemma}
\begin{proof}
	Set $y=x/u$.
	We have $y \in \mathfrak m$ because $\myval(y)>e$.
	Since $y \in \mathfrak m$, the elements $\dfrac{y+1}{2}$ and $\dfrac{y-1}{2}$ are elements in $A$.
	We have $$x = uy = \left(\dfrac{y+1}{2}\right)^2u+\left(\dfrac{y-1}{2}\right)^2\cdot(-u) \in \mathcal M.$$
\end{proof}

\begin{lemma}\label{lem:basic4}
	Let $A$ be a pseudo-valuation domain such that any strict units admit a square root.
	Let $x_1$ and $x_2$ be nonzero elements in $A$.
	If $\myval(x_1)=\myval(x_2)$ and $\mylcp(x_1)=\mylcp(x_2)$, we have $x_2=x_1u^2$ for some $u \in A^\times$ with $\pi(u^2)=1$.
\end{lemma}
\begin{proof}
	There exists $u \in B^\times$ satisfying the equalities in the lemma by \cite[Lemma 2.15]{FK}.
	We have $\pi(u)=\pm 1 \in F_0$.
	Lemma \ref{lem:basic} implies that $u \in A^\times$.
\end{proof}

\begin{lemma}\label{lem:basic5}
	Let $A$ be a pseudo-valuation domain such that any strict units admit a square root.
	Let $\mathcal M$ be a quasi-quadratic $A$-module in $A$.
	Let $x$ be a nonzero element in $\mathcal M$ and $y \in A$ with $y \neq 0$.
	We assume that $\myval(x)=\myval(y)$ and $\mylcp(x)=\mylcp(y)$.
	Then, we have $y \in \mathcal M$.
\end{lemma}
\begin{proof}
	There exists $u\in A^\times$ such that $y=u^2x$ by Lemma \ref{lem:basic4}.
	It implies that $y \in \mathcal M$.
\end{proof}

We are ready to demonstrate that both $M_g(\mathcal M)$ and $\Psi(M,g)$ are quasi-quadratic modules.

\begin{proposition}\label{prop:M}
	Let $A$ be a pseudo-valuation domain such that any strict units admit a square root.
	The set $M_g(\mathcal M)$ is a quasi-quadratic $F_0$-module in $F$ whenever $g \in G_{\geq e}$ and $\mathcal M$ is a quasi-quadratic $A$-module in $A$. 
\end{proposition}
\begin{proof}
	Set $M=M_g(\mathcal M)$ for simplicity.
	The proposition is obvious when $g \not\in \myval(\mathcal M)$.
	We consider the case in which $g \in \myval(\mathcal M)$.
	We have already demonstrated that $M$ is closed under multiplication by the squares of elements in $F_0$ in Lemma \ref{lem:basic2}.
	We have only to show that $M$ is closed under addition.
	
	Let $a,b$ be arbitrary elements in $M$.
	It is obvious that $a+b \in M$ when $a=0$, $b=0$ or $a+b=0$.
	Let us consider the other case.
	Take $x,y \in \mathcal M$ with $\myval(x)=\myval(y)=g$, $a=\mylcp(x)$ and $b=\mylcp(y)$.
	We have $\myval(x+y)=g$ and $a+b=\mylcp(x)+\mylcp(y)=\mylcp(x+y)$ by Definition \ref{def:pseudo}(4).
	It means that $a+b \in M$.
\end{proof}

\begin{proposition}\label{prop:Psi}
	Let $A$ be a pseudo-valuation domain such that $2$ is a unit and any strict units admit a square root.
	The set $\Psi(M,g)$ is a quasi-quadratic $A$-module in $A$ whenever $g \in G_{\geq e}$ and $M$ is a quasi-quadratic $F_0$-module in $F$. 
\end{proposition}
\begin{proof}
	Set $\Psi_i=\Psi_i(M,g)$ for $i=1,2$ and $\Psi=\Psi(M,g)$ for simplicity.
	Note that $\myval(\Psi) \subseteq G_{\geq g} \cup \{\infty\}$.
	
	We first demonstrate that $\Psi$ is closed under addition.
	Take $x,y \in \Psi$.
	It is obvious that $x+y \in \Psi$  when $x=0$, $y=0$ or $x+y=0$.
	We consider the other case.
	Consider the case in which $\myval(x) \neq \myval(y)$.
	We may assume that $\myval(x)<\myval(y)$.
	We have $\myval(x+y)=\myval(x)$ and $\mylcp(x+y)=\mylcp(x)$ by Definition \ref{def:pseudo}(4).
	It is easy to check that $x+y \in \Psi_i$ when $x \in \Psi_i$ for $i=1,2$ in this case.
	We omit the details.
	
	The next target is the case in which $\myval(x)=\myval(y)$.
	We first consider the case in which $x \in \Psi_1 \setminus \Psi_2$.
	By the definition of $\Psi_2$, we get $y \not\in  \Psi_2$.
	It implies that $y$ also belongs to $\Psi_1$.
	When $\mylcp(x)+\mylcp(y)=0$, it follows that $\mylcp(-x)=-\mylcp(x)=\mylcp(y)$.
	It is obvious that $\myval(-x)=\myval(x)$. 
	Therefore, the element $-x$ also belongs to $\Psi_1$ both when $\myval(y)=g$ and $\myval(y)>g$.
	On the other hand, we have $\myval(x+y)>\myval(x)$ by Lemma \ref{lem:lcp_basic2}.
	They imply that $x+y \in \Psi_2 \subseteq \Psi$.
	
	When $\mylcp(x)+\mylcp(y) \neq 0$, Definition \ref{def:pseudo}(4) implies that $\myval(x+y)=\myval(x)$ and $\mylcp(x+y)=\mylcp(x)+\mylcp(y)$.
	They imply that $x+y \in \Psi_1$ both when $\myval(x)=g$ and $\myval(x)>g$.
	
	We next consider the case in which $x \in \Psi_2$.
	Since $\myval(x+y) \geq \myval(x)$, the sum $x+y$ also belongs to $\Psi_2$.
	We have proven that $\Psi$ is closed under addition.
	
	We show that $\Psi$ is closed under multiplication by the square of an element in $A$.
	Take $a \in A$ and $x \in \Psi$.
	It is obvious that $a^2x \in \Psi$ when $a=0$ or $x=0$.
	Assume that $a$ and $x$ are not zeros.
	We first treat the case in which $x \in \Psi_1$.
	When $\myval(x)=g$ and $a \in A^\times$, we get $\myval(a^2x)=\myval(x)=g$ and $\mylcp(a^2x)=\pi_0(a)^2\mylcp(x) \in M$ by Definition \ref{def:pseudo}(2) because $M$ is a quasi-quadratic $F_0$-module in $F$.
	In the other case, we obtain $\myval(a)>e$ or $\myval(x)>g$.
	We have $\myval(a^2x)>g$.
	There exists $k \in F$ such that $\mylcp(a^2x)=k^2\mylcp(x)$.
	Since $\mylcp(x) \in \mycl_F(M)$, we get $\mylcp(a^2x) \in \mycl_F(M)$.
	It together with the equality $\overline{\myval(a^2x)}=\overline{g}$ implies that $a^2x \in \Psi_1$.
	
	The remaining case is the case in which $x \in \Psi_2$.
	However, it immediately follows from the definition of $\Psi_2$ that $a^2x \in \Psi_2$ because $\myval(a^2x) \geq \myval(x)$. 
\end{proof}

The following structure theorem says that a quasi-quadratic $A$-module in $A$ is the union of quasi-quadratic modules of the form $\Psi(M,g)$.
\begin{theorem}\label{thm:decomp}
	Let $A$ be a pseudo-valuation domain such that $2$ is a unit and any strict units admit a square root.
	Consider a quasi-quadratic $A$-module $\mathcal M$ in $A$.
	The equality $$\mathcal M = \bigcup_{g \in G_{\geq e}}\Psi(M_g(\mathcal M),g)$$ holds true.
\end{theorem}
\begin{proof}
	Set $\mathcal N=\bigcup_{g \in G_{\geq e}}\Psi(M_g(\mathcal M),g)$ for simplicity.
	The inclusion $\mathcal M \subseteq \mathcal N$ is obvious.
	We demonstrate the opposite inclusion.
	Take $x \in \mathcal N$.
	We obviously have $x \in \mathcal M$ when $x=0$.
	So, we concentrate on the other case.
	We have $x \in \Psi(M_g(\mathcal M),g)$ for some $g \in G_{\geq e}$.
	Set $M=M_g(\mathcal M)$ and $\Psi=\Psi(M,g)$.
	We also define $\Psi_1$ and $\Psi_2$ in the same manner.
	
	When $x \in \Psi_1$ and $\myval(x)=g$, there exists $y \in \mathcal M$ such that $\myval(x)=\myval(y)$ and $\mylcp(x)=\mylcp(y)$.
	It implies that $x \in \mathcal M$ by Lemma \ref{lem:basic5}.
	
	We then consider the case in which $x \in \Psi_1$ and $\myval(x)>g$.
	We have $\myval(x)=gh^2$ for some $h \in G_{\geq e}$ by the definition of $\Psi_1$.
	There exist a positive integer $n$, elements $u_1, \ldots, u_n \in F$ and $v_1, \ldots, v_n \in M$ such that $\mylcp(x)=\sum_{i=1}^n u_i^2v_i$.
	We may assume that, for any $1 \leq j \leq n$, (*) $\sum_{i=1}^ju_i^2v_i \neq 0$ without loss of generality.
	We can take $y_i \in \mathcal M$ with $\myval(y_i)=g$ and $v_i=\mylcp(y_i)$ for each $1 \leq i \leq n$ by the definition of $M$. 
	We can choose $w_i \in K^\times$ such that $\myval(w_i)=h$ and $\mylcp(w_i^2y_i)=u_i^2v_i$ for each $1 \leq i \leq n$ by Lemma \ref{lem:lcp_basic1}.
	Since $\myval(w_i)=h>e$, we get $w_i \in \mathfrak m \subset A$.
	Put $z=\sum_{i=1}^n w_i^2y_i$, then $z \in \mathcal M$.
	On the other hand, applying Definition \ref{def:pseudo}(4) inductively together with the assumption (*), we get $\myval(z)=gh^2=\myval(x)$ and $\mylcp(z)=\sum_{i=1}^nu_i^2v_i=\mylcp(x)$.
	It implies that $x \in \mathcal M$  by Lemma \ref{lem:basic5}.
	
	The remaining case is the case in which $x \in \Psi_2$.
	By the definition of $\Psi_2$, there exists $u \in \Psi_1$ such that $\pm u\in\Psi_1$ and $\myval(u)<\myval(x)$.
	We have already demonstrated that any element in $\Psi_1$ belongs to $\mathcal M$.
	In particular, we have $\pm u \in \mathcal M$.
	We immediately get $x \in \mathcal M$ by Lemma \ref{lem:basic3}.
\end{proof}

\begin{proposition}\label{supp1}
	Let $A$ be a pseudo-valuation domain such that $2$ is a unit and any strict units admit a square root.
	Consider a quasi-quadratic $A$-module $\mathcal M$ in $A$ and a nonzero element $x\in A$ with $\myval(x)=g$.
	Then the following conditions are equivalent:
	\begin{enumerate}
		\item[(1)] $\pm x\in \mathcal M$;
		\item[(2)] $\pm\mylcp(x)\in M_g(\mathcal M)$.
	\end{enumerate}
	Furthermore, $M_h(\mathcal M)=F$ whenever  $\pm\mylcp(x)\in M_g(\mathcal M)$ and $g<h$.
\end{proposition}
\begin{proof}
	(1) $\Rightarrow$ (2): It is nothing to prove.
	
	(2) $\Rightarrow$ (1): It is immediate from Lemma \ref{lem:basic5}.
	
	We next demonstrate the `furthermore' part. Take a nonzero element $c\in F$. 
	We want to show that $c\in M_h(\mathcal M)$. By Definition \ref{def:pseudo}(3), there exists an element $u\in K$ with
	$\mylcp(u)=c$ and $\myval(u)=h$. Since $\myval(u)>g\geq e$, we have $u\in \mathfrak m\subset A$.
	It follows from the assumption that $\pm x\in \mathcal M$.
	Hence we have $u\in \mathcal M$ by Lemma \ref{lem:basic3}.
	This implies that $c=\mylcp(u)\in M_h(\mathcal M)$, as desired.
\end{proof}

\begin{corollary}\label{supp2}
	Let $A$ be a pseudo-valuation domain such that $2$ is a unit and any strict units admit a square root.
	Consider a quasi-quadratic $A$-module $\mathcal M$ in $A$.
	Then the following conditions are equivalent:
	\begin{enumerate}
		\item[(1)] $\mathcal M =A$;
		\item[(2)] $M_g(\mathcal M)=F$ for any $g\in G_{>e}$ and $M_e(\mathcal M)=F_0$;
		\item[(3)] $M_e(\mathcal M)=F_0$.
	\end{enumerate}
\end{corollary}
\begin{proof}
	(1) $\Rightarrow$ (2): By the assumption, we have $\pm 1\in M_e(\mathcal M)$.
	It follows from Proposition \ref{supp1} that $M_g(\mathcal M)=F$ for any $g>e$ and $\pm 1\in \mathcal M$.
	Note that $M_e(\mathcal M)$ is a quasi-quadratic $F_0$-module in $F_0$ because we have $\mylcp(x)=\pi(x) \in F_0$ for any nonzero element $x \in A$ with $\myval(x)=e$.
	Take an element $d\in F_0$. We have
	$d=\Bigl(\dfrac{d+1}{2}\Bigl)^2\cdot 1+\Bigl(\dfrac{d-1}{2}\Bigl)^2\cdot (-1)\in M_e(\mathcal M)$.
	This implies that $M_e(\mathcal M)=F_0$.

	(2) $\Rightarrow$ (3): This requires no comment.

	(3) $\Rightarrow$ (1): We get $\pm 1\in M_e(\mathcal M)$. By Proposition \ref{supp1}, 
	it follows that $\pm 1\in\mathcal M$.
	Hence we have
	$a=\Bigl(\dfrac{a+1}{2}\Bigl)^2\cdot 1+\Bigl(\dfrac{a-1}{2}\Bigl)^2\cdot (-1)\in \mathcal M$
	for any $a\in A$.
\end{proof}

\begin{corollary}\label{supp3}
	Let $A$ be a pseudo-valuation domain such that $2$ is a unit and any strict units admit a square root.
	Consider a quasi-quadratic $A$-module $\mathcal M$ in $A$.
	Then the following assertions hold:
	\begin{enumerate}
		\item[(1)] $\mathcal M \cap (-\mathcal M)=\{x\in A\setminus \{0\}\;\vert\;\pm\mylcp(x)\in M_{\myval(x)}(\mathcal M)\}
		\cup\{0\}$.
		\item[(2)] The quasi-quadratic $A$-module $\mathcal M$ becomes
		an ideal of $A$ if and only if $\pm\mylcp(x)\in M_{\myval(x)}(\mathcal M)$ for any nonzero $x\in \mathcal M$. 
	\end{enumerate}
\end{corollary}
\begin{proof}
	(1) It is obvious from Proposition \ref{supp1}.
	
	(2) Assume that $\mathcal M$ is an ideal of $A$. Since $\mathcal M \cap (-\mathcal M)=\mathcal M$, 
	the `only if' part follows from (1). We demonstrate the opposite implication.
	Let $x$ be a nonzero element of $\mathcal M$. 
	It is enough to prove that $ax\in \mathcal M$ for any $a\in A$.
	By the assumption we get $-\mylcp(x)\in M_{\myval(x)}(\mathcal M)$.
	Hence it follows from Lemma \ref{lem:basic5} that $-x\in\mathcal M$.
	Thus we have
	$ax=\Bigl(\dfrac{a+1}{2}\Bigl)^2\cdot x+\Bigl(\dfrac{a-1}{2}\Bigl)^2\cdot (-x)\in \mathcal M$.
\end{proof}

We can show that an ideal of a pseudo-valuation ring is represented as a union of
sets of the form $\Psi(F,g)$ for $g\in G_{\geq e}$. To see this, 
we first note that the following equalities hold true:
\begin{align*}
	\Psi_1(F,g)&=\{x \in A \setminus \{0\} \mid \overline{\myval(x)} = \overline{g} \text{ and }\myval(x)\geq g \},\\
	\Psi_2(F,g)&=\{x \in A \setminus \{0\} \mid  \exists u \in A\setminus\{0\}
	\text{ with }\overline{\myval(u)}=\overline{g} \text{ and } g\leq \myval(u)<\myval(x)\}.
\end{align*}

\begin{corollary}\label{ideal-1}
	Let $A$ be a pseudo-valuation domain such that $2$ is a unit and any strict units admit a square root.
	Let $I$ be an ideal of $A$. Set $S=\myval(I)$. Then the following assertions hold true:
	\begin{enumerate}[(1)]
		\item If the set $S$ has no smallest element, we have
		$$
		I=\bigcup_{x\in I\setminus\{0\}}\Psi(F,\myval(x))\cup\{0\}.
		$$
		\item If the set $S$ has the smallest element $m(I)$,
		we have
		$$
		I=\Psi(M_{m(I)}(I),m(I)).
		$$
	\end{enumerate}
\end{corollary}
\begin{proof}
	(1) Assume that $I\neq \{0\}$. It can be easily seen that the ideal $I$
	is included the right hand side of the equality. We next show the opposite inclusion.
	Take a nonzero element $y\in \Psi(F,\myval(x))$ for some $x\in I\setminus\{0\}$.
	We first consider the case in which $y\in \Psi_1(F,\myval(x))$. It follows that
	$\overline{\myval(y)}=\overline{\myval(x)}$ and $\myval(y)\geq \myval(x)$.
	Since $S$ has no smallest element, there exists an element $w\in I$ such that
	$\myval(x)>\myval(w)$.
	It follows from Proposition \ref{supp1} that $M_{\myval(y)}(I)=F$ because 
	$\pm\mylcp(w)\in M_{\myval(w)}(I)$ by Corollary \ref{supp3}.
	Since $\mylcp(y)$ belongs to $M_{\myval(y)}(I)$, there exists an element $z\in I$ such that
	$\mylcp(y)=\mylcp(z)$ and $\myval(y)=\myval(z)$. Hence we have
	$y\in I$ from Lemma \ref{lem:basic5}.
	
	We next consider the remaining case in which $y\in \Psi_2(F, \myval(x))$.
	There exists a nonzero $u\in A$ such that $\myval(x)\leq\myval(u)<\myval(y)$.
	Thus we have $M_{\myval(y)}(I)=F$ by Proposition \ref{supp1} because $\pm \mylcp(x) \in M_{\myval(x)}(I)$. 
	We get $y\in I$ by Lemma \ref{lem:basic5} and the definition of $M_{\myval(y)}(I)$ because $\mylcp(y) \in M_{\myval(y)}(I)$. We have shown the assertion (1).
	
	(2) By Theorem \ref{thm:decomp}, we see that
	the ideal $I$ includes the right hand side of the equality.
	We show the opposite inclusion. Take a nonzero element $x\in I$.
	There exists an element $x_0\in I$ such that $m(I)=\myval(x_0)$.
	When $\myval(x)=\myval(x_0)$, it is immediate that $x\in\Psi_1(M_{m(I)}(I),m(I))\subseteq
	\Psi(M_{m(I)}(I),m(I))$. We next suppose that $\myval(x)>\myval(x_0)$.
	It follows from Corollary \ref{supp3}(2) that $\pm x_0\in\Psi_1(M_{m(I)}(I),m(I))$.
	Thus we have $x\in \Psi_2(M_{m(I)}(I),m(I))\subseteq\Psi(M_{m(I)}(I),m(I))$.
\end{proof}

The set $\mathfrak X_F$ is naturally a subset of $\mathfrak X_{F_0}^F$.
We want to show that there is a bijection between $\mathfrak X_A$ and 
\begin{align*}
	\mathfrak S_{F_0}^F&:=\Bigl\{(M_g)_{g \in G_{\geq e}} \in \prod_{g \in G_{\geq e}} \mathfrak X_{F_0}^F \Bigm\vert  M_g \text{ satisfies }\\
	&\Bigl.\text{ \qquad the following conditions (i) through (iii)}\Bigr\}.
\end{align*}
\begin{enumerate}[(i)]
	\item $M_e \subseteq F_0$;
	\item $\mycl_F(M_g) \subseteq M_h$ whenever $e \leq g < h$ and $\overline{g}=\overline{h}$;
	\item $M_h=F$ whenever there exist $g \in G_{\geq e}$ and $u \in F\setminus\{0\}$ with $e \leq g<h$ and $\pm u \in M_g$.
\end{enumerate}

For that purpose, we first demonstrate the following proposition:
\begin{proposition}\label{prop:sigma_inverse}
	Let $A$ be a pseudo-valuation domain such that $2$ is a unit and any strict units admit a square root.
	The union $\bigcup_{g\in G_{\geq e}}\Psi(M_{g},g)$ is a quasi-quadratic $A$-module in $A$ for each $(M_g)_{g \in G_{\geq e}} \in \mathfrak S_{F_0}^F$.
\end{proposition}
\begin{proof}
	We have only to demonstrate that $\mathcal M=\bigcup_{g\in G_{\geq e}}\Psi(M_{g},g)$ is closed under addition and multiplication by the square of elements in $A$.
	The closedness under multiplication by the square immediately follows from Proposition \ref{prop:Psi}.
	
	We demonstrate the closedness under addition.
	Let $x_1,x_2$ be elements in $\mathcal M$.
	We prove that $x_1+x_2 \in \mathcal M$.
	It is obvious when at least one of $x_1$, $x_2$ and $x_1+x_2$ is zero.
	So we assume that they are nonzero elements.
	Take $g_1,g_2 \in G_{\geq e}$ with $x_i \in \Psi(M_{g_i},g_i)$ for $i=1,2$.
	When $\myval(x_1) \neq \myval(x_2)$, we may assume that $\myval(x_1)<\myval(x_2)$.
	We have $\myval(x_1+x_2)=\myval(x_1)$ and $\mylcp(x_1+x_2)=\mylcp(x_1)$ by Definition \ref{def:pseudo}(4).
	These equalities imply that $x_1+x_2 \in \Psi(M_{g_1},g_1)$.
	
	We next consider the case in which $\myval(x_1)=\myval(x_2)$.
	When $g_1=g_2$, we immediately get $x_1+x_2 \in  \Psi(M_{g_1},g_1)$ by Proposition \ref{prop:Psi}.
	Consider the case in which $g_1 \neq g_2$.
	We demonstrate that $x_i \in \Psi(M_{g_{3-i}},g_{3-i})$ for some $i=1,2$.
	If it is true, we obviously have $x_1+x_2 \in \Psi(M_{g_{3-i}},g_{3-i})$ by Proposition \ref{prop:Psi}.
	The proof has been completed.
	
	Set $h=\myval(x_1)=\myval(x_2)$.
	We first consider the case in which $x_1 \in \Psi_1(M_{g_1},g_1)$ and $x_2 \in \Psi_1(M_{g_2},g_2)$.
	We may assume that $g_1<g_2$ without loss of generality.
	Since $\mycl_F(M_{g_1}) \subseteq M_{g_2}$ by the assumption, we have $\mylcp(x_1) \in M_{g_2}$.
	It implies that $x_1 \in \Psi_1(M_{g_2},g_2) \subseteq \Psi(M_{g_2},g_2)$.
	The remaining case is the case in which at least one of $x_1 \in \Psi_2(M_{g_1},g_1)$ and $x_2 \in \Psi_2(M_{g_2},g_2)$ holds.
	We may assume that $x_1 \in \Psi_2(M_{g_1},g_1)$ without loss of generality.
	By the definition of $\Psi_2(M_{g_1},g_1)$, there exists $u \in \Psi_1(M_{g_1},g_1)$ with $\pm u \in \Psi_1(M_{g_1},g_1)\subseteq \Psi(M_{g_1},g_1)$ and $\myval(u)<h$.
	Recall that $\myval(x_2)=h$ by the assumption.
	We have $x_2 \in \Psi_2(M_{g_1},g_1) \subseteq \Psi(M_{g_1},g_1)$ using the definition of $\Psi_2(M_{g_1},g_1)$ again.
\end{proof}

We are now ready to demonstrate our main theorem.
\begin{theorem}\label{thm:wts1}
	Let $A$ be a pseudo-valuation domain such that $2$ is a unit and any strict units admit a square root.
	Consider a quasi-quadratic $A$-module $\mathcal M$ in $A$.
	We define the map 
	$\sigma:\mathfrak X_A \rightarrow \mathfrak  S_{F_0}^F$ by
	\begin{equation*}
		\sigma(\mathcal M)=(M_{g}(\mathcal M))_{g \in G_{\geq e}}.
	\end{equation*}
	The map $\sigma$ is a bijection.
\end{theorem}
\begin{proof}
	We first show that $\sigma$ is well-defined. 
	We have to check that the family $\sigma(\mathcal M)$ satisfies the conditions (i) through (iii) in the definition of $\mathfrak  S_{F_0}^F$.
	Satisfaction of condition (i) immediately follows from Definition \ref{def:pseudo}(1), the definition of $M_e(\mathcal M)$ and the fact that $\pi(A) =F_0$.
	We consider the condition (ii).
	By Proposition \ref{prop:M}, $M_g(\mathcal M)$ is a
	quasi-quadratic $F_0$-module in $F$ for any $g\in G_{\geq e}$.
	Take elements $g, h\in G_{\geq e}$ with $g<h$ and $\overline{g}=\overline{h}$.
	There exists an element $h_1\in G_{>e}$ such that $h=gh_1^2$.
	It follows from Lemma \ref{lem:basic2} that $F^2c \subseteq M_{gh_1^2}(\mathcal M)=M_h(\mathcal M)$ 
	for any $c\in M_g(\mathcal M)$.
	This implies that $\mycl_F(M_g(\mathcal M))\subseteq M_h(\mathcal M)$. 
	Next we check that $\sigma(\mathcal M)$ satisfies the condition (iii).
	Assume that there exist $g,h\in G_{\geq e}$ and $u\in F$ with $g<h$ and $\pm u\in M_g(\mathcal M)$.
	There exists an element $x\in\mathcal M$ such that $\mylcp(x)=u$ and $\myval(x)=g$. Since $\pm\mylcp(x)=\pm u\in M_g(\mathcal M)$,
	we have $M_h(\mathcal M)=F$ from Proposition \ref{supp1}, as desired.
	
	
	We define the map $\rho:\mathfrak S_{F_0}^F \rightarrow \mathfrak X_{A}$ by
	\begin{equation*}
		\rho\left((M_{g})_{g\in G_{\geq e}}\right)
		=\bigcup_{g\in G_{\geq e}}\Psi(M_{g},g)\text{.}
	\end{equation*}
	The map $\rho$ is well-defined by Proposition \ref{prop:sigma_inverse}.
	The composition $\rho \circ \sigma$ is the identity map by Theorem \ref{thm:decomp}.
	We demonstrate that $\sigma \circ \rho$ is also the identity map.
	Fix $g \in G_{\geq e}$.
	Let $N$ be the $g$-th coordinate of $\sigma \circ \rho\left((M_g)_{g\in G_{\geq e}}\right)$.
	We want to show that $N=M_{g}$.
	We first demonstrate the inclusion $N \subseteq M_g$.
	We have 
	\begin{align*}
			N &= M_{g}\left(\bigcup_{h\in G_{\geq e}}\Psi(M_{h}, h)\right)\\
			&=\Bigl\{\mylcp(x)\;\vert\; x \in \bigcup_{h\in G_{\geq e}}\Psi(M_{h},h)
			\setminus\{0\}
			\text{ and } \myval(x)=g\Bigr\} \cup \{0\}.
	\end{align*}
	Take an arbitrary nonzero element $x \in A$ with $\mylcp(x) \in N$ and $\myval(x)=g$.
	We also take $h \in G_{\geq e}$ such that $x \in \Psi(M_{h},h)$.
	If $x \in \Psi_2(M_{h},h)$, there exists $y\in\Psi_1(M_h,h)$ with
	$\pm y \in\Psi_1(M_h,h)$ and $\myval(y)<\myval(x)$. We first consider the case in which
	$\myval(y)=h$. It follows that $\pm\mylcp(y)\in M_h$.
	Since $h=\myval(y)<\myval(x)=g$, we have $M_g=F$ from the condition (iii)
	in the definition of $\mathfrak{S}_{F_0}^F$.
	The inclusion $N \subseteq M_g$ is obvious in this case.
	Next  we demonstrate the remaining case in which $\myval(y)>h$.
	Recall that we have $\overline{\myval(y)}=\overline{h}$ from the definition of $\Psi_1(M_h,h)$.
	We get $\pm\mylcp(y)\in\mycl_F(M_h)\subseteq M_{\myval(y)}$ by the condition (ii).
	It together with $\myval(y)<g$ and the condition (iii) implies that $M_g=F$, as desired.
	When $x \in \Psi_1(M_{h},h)$, we have $\mylcp(x) \in M_h=M_g$ if $g=h$ and $\mylcp(x) \in \mycl_F(M_h) \subseteq M_g$ otherwise.
	
	We demonstrate the opposite inclusion.
	Take a nonzero element $c \in M_{g}$.
	There exists a nonzero element $w \in B$ with $\myval(w)=g$ and $\mylcp(w)=c$ by Definition \ref{def:pseudo}(3). 
	We have $\pi(w)=c \in M_e \subseteq F_0$ by the condition (i) when $g=e$ and $\pi(w)=0 \in F_0$ otherwise.
	They imply that $w \in A$ by Lemma \ref{lem:basic}.
	Since $w\in\Psi(M_g,g)$, we have $c \in N$.
	We have finished to prove that $\sigma$ and $\rho$ are the inverses of the others.
\end{proof}

We next consider the intersection and the sum of two quasi-quadratic modules.
\begin{theorem}\label{thm:wts2}
	Let $A$ be a pseudo-valuation domain such that $2$ is a unit and any strict units admit a square root.
	Consider quasi-quadratic $A$-modules $\mathcal M_1$ and $\mathcal M_2$ in $A$.
	Set $M_{1,g}=M_g(\mathcal M_1)$ and $M_{2,g}=M_g(\mathcal M_2)$ for $g \in G_{\geq e}$.
	For each $g \in G_{\geq e}$, the following equalities hold true:
	\begin{enumerate}
		\item[(1)] $M_g(\mathcal M_1 \cap \mathcal M_2)=M_{1,g} \cap M_{2,g}$.\medskip
		\item[(2)] $M_g(\mathcal M_1+\mathcal M_2)$
						\begin{align*}
								&= \left\{\begin{array}{ll}
											F & \text{when } \exists h \in G_{\geq e}  \text{with } h<g \\
											& \quad \text{and } M_{1,h} \cap (-M_{2,h}) \neq \{0\},\\
											M_{1,g}+M_{2,g} & \text{elsewhere.}
										\end{array}\right.
								\end{align*}
	\end{enumerate}
\end{theorem}
\begin{proof}
	(1) 
	The inclusion $M_g(\mathcal M_1 \cap \mathcal M_2) \subseteq M_{1,g} \cap M_{2,g}$ is obvious.
	We demonstrate the opposite inclusion.
	Take an arbitrary nonzero element $a \in M_{1,g} \cap M_{2,g}$.
	Since $a \in M_{i,g}$, there exists $w_i \in \mathcal M_i$ such that $\myval(w_i)=g$ and $\mylcp(w_i)=a$ for $i=1,2$.
	Lemma \ref{lem:basic5} implies that $w_2$ is an element of $\mathcal M_1$.
	We have $a \in M_g(\mathcal M_1 \cap \mathcal M_2)$ because $w_2 \in \mathcal M_1 \cap \mathcal M_2$.
	
	(2)
	We first consider the case in which there exists $h \in G_{\geq e}$ with $h<g$ and $M_{1,h} \cap (-M_{2,h}) \neq \{0\}$.
	Take a nonzero element $a \in M_{1,h} \cap (-M_{2,h}) $.
	There exist $w_1 \in \mathcal M_1$ and $w_2 \in \mathcal M_2$ with $\myval(w_1)=\myval(w_2)=h$ and $\mylcp(w_1)=-\mylcp(w_2)=a$.
	By Lemma \ref{lem:basic5}, we get $-w_2  \in \mathcal M_1$.
	It implies that $\pm w_2$ belongs to $\mathcal M_1+\mathcal M_2$.
	This means that $\pm w_2\in \mathcal M_1+\mathcal M_2$. Thus it follows from Proposition \ref{supp1} that
	$M_g(\mathcal M_1+\mathcal M_2)=F$.
	
	We next consider the remaining case.
	The inclusion $M_{1,g}+M_{2,g} \subseteq M_g(\mathcal M_1+\mathcal M_2)$  immediately follows from Definition \ref{def:pseudo}(4).
	We demonstrate the opposite inclusion.
	Take a nonzero element $a \in M_g(\mathcal M_1+\mathcal M_2)$ and an element $w \in \mathcal M_1+\mathcal M_2$ with $\myval(w)=g$ and $\mylcp(w)=a$ as usual.
	There exist $w_1 \in \mathcal M_1$ and $w_2 \in \mathcal M_2$ with $w=w_1+w_2$.
	It is obvious that $a \in M_{1,g}+M_{2,g}$ when either $w_1=0$ or $w_2=0$.
	Therefore, we assume that both $w_1$ and $w_2$ are nonzero.
	
	We may assume that $\myval(w_1) \leq \myval(w_2)$ without loss of generality. Set $h=\myval(w_1)$.
	We have $g=\myval(w_1+w_2)\geq\min(\myval(w_1), \myval(w_2))=h$. 
	Assume for contradiction that $h<g$.
	We have $\myval(w_1)=\myval(w_2)$ because $\myval(w_1+w_2)>\myval(w_1)$.
	We also have $\mylcp(w_2)=-\mylcp(w_1)$ by Definition \ref{def:pseudo}(4).
	We get $-w_1 \in \mathcal M_2$ by Lemma \ref{lem:basic5}.
	We have $\mylcp(w_1) \in M_{1,h} \cap (-M_{2,h})$ and it is a contradiction to the assumption that $M_{1,h} \cap (-M_{2,h})=\{0\}$.
	We have demonstrated that $h=g$.
	We have either $\myval(w_1)<\myval(w_2)$ or $\myval(w_1)=\myval(w_2)$.
	In the former case, we obviously have $a=\mylcp(w)=\mylcp(w_1+w_2)=\mylcp(w_1) \in M_{1,g} \subseteq M_{1,g}+M_{2,g}$ by Definition \ref{def:pseudo}(4).
	In the latter case, we have $\mylcp(w_1)+\mylcp(w_2) \neq 0$ by Lemma \ref{lem:lcp_basic2} and the equality $\myval(w_1)=\myval(w_1+w_2)$.
	We have $a=\mylcp(w)=\mylcp(w_1)+\mylcp(w_2) \in M_{1,g}+M_{2,g}$ by Definition \ref{def:pseudo}(4).
\end{proof}

\section{Special cases.}\label{sec:example}

\subsection{Examples.}\label{subsec:example}

We consider the correspondence between quasi-quadratic modules and cones in some special cases.

\begin{example}\label{ex:ex1}
	Let $(E,<)$ be a Euclidean field, which is an ordered field in which a positive element is a square.
	Let $V$ be an $E$-vector space.
	A subset $S$ of $V$ is \textit{convex} if $\lambda x + (1-\lambda)y \in S$ whenever $x,y \in S$ and $\lambda \in E$ with $0 \leq \lambda \leq 1$.
	A subset $S$ of $V$ is a \textit{ray} if $ax \in S$ whenever $a$ is a nonnegative element of $E$ and $x \in S$.
	
	A subset $Q$ of $V$ is a quasi-quadratic $E$-module in $V$ if and only if it is a convex ray.
	It is obvious that $\sum E^2=E_{\geq 0}$.
	Using this fact, the proof of the above equivalence is not so difficult.
	
	In fact, let $Q$ be a quasi-quadratic $E$-module in $V$. We first show that $Q$ is a ray. Take a positive element $a\in E$ and an element $x\in Q$.
	There exists an element $b\in E$ such that $a=b^2$. Since $Q$ is a quasi-quadratic $E$-module of $V$, we have
	$ax\in Q$. Next we show that $Q$ is convex. Take elements $x,y\in Q$. 
	We have $1\cdot x+0\cdot y=x\in Q$ and $0\cdot x+1\cdot y=y\in Q$. Take an element $\lambda\in E$ with $0<\lambda<1$. 
	Since $\lambda\in E^2$ and $1-\lambda \in E^2$, we get $\lambda x\in Q$ and $(1-\lambda)y\in Q$. 
	Hence we have $\lambda x+(1-\lambda)y\in Q$.
	
	We demonstrate the opposite implication. We assume that $Q$ is a convex ray. Take elements $x,y\in Q$.
	It follows that $a^2 x\in Q$ for any $a\in E$ because $a^2$ is nonnegative.
	Since $Q$ is convex and $0<\dfrac{1}{2}<1$ in $E$, we get $\dfrac{1}{2}x+\Bigl(1-\dfrac{1}{2}\Bigl)y=\dfrac{x+y}{2}\in Q$.
	Thus we have $x+y\in Q$ because the element $2$ is positive in $E$ and $Q$ is a ray. 
\end{example}

Let us recall a fundamental example introduced in Section 1.

\begin{example}\label{ex:ex2}
	We consider the pseudo-valuation domain $$A=\left\{\left.\sum_{l=0}^\infty a_lx^l \in \mathbb C[\![X]\!]\;\right\vert\; a_0 \in \mathbb R\right\}.$$
	We have $B=\mathbb C[\![X]\!]$, $F=\mathbb C$, $F_0=\mathbb R$ and $G=(\mathbb Z,+)$ in this case.
	The pseudo-angular component map $\mylcp$ is defined by taking the coefficient of the term of minimum degree.
	Since $\mathbb C$ is algebraically closed, any element is the square of some element.
	Therefore, quasi-quadratic $\mathbb C$-modules in $\mathbb C$ are trivial ones, that is, $\{0\}$ and $\mathbb C$.
	Under the identification of $\mathbb C$ with $\mathbb R\times \mathbb R$, only three types of quasi-quadratic $\mathbb R$-modules in $\mathbb C$ are possible by Example \ref{ex:ex1}.
	\begin{enumerate}[(i)]
		\item Trivial quasi-quadratic modules, that is, $\{0\}$ and $\mathbb C$;
		\item Lines. For any $z \in \mathbb C$ with $z \neq 0$, set $M_{\text{line}}(z)=\{tz \in \mathbb C\;\vert\; t \in \mathbb R\}$;
		\item Convex fans. For any $\theta_1,\theta_2$ with $0 \leq \theta_2 \leq \pi$, set
				\begin{align*}
						&M_{\text{fan,oo}}(\theta_1,\theta_2)=\{0\} \cup \{re^{i\theta} \in \mathbb C\;\vert\; r>0,
						\theta_1 < \theta < \theta_1+\theta_2\}, \\
						&M_{\text{fan,oc}}(\theta_1,\theta_2)=\{0\} \cup \{re^{i\theta} \in \mathbb C\;\vert\; r>0,
						\theta_1 < \theta \leq \theta_1+\theta_2\}, \\
						&M_{\text{fan,co}}(\theta_1,\theta_2)=\{0\} \cup \{re^{i\theta} \in \mathbb C\;\vert\; r>0,
						 \theta_1 \leq \theta < \theta_1+\theta_2\}, \\
						&M_{\text{fan,cc}}(\theta_1,\theta_2)=\{0\} \cup \{re^{i\theta} \in \mathbb C\;\vert\; r>0,
						 \theta_1 \leq \theta \leq \theta_1+\theta_2\}.
					\end{align*} 
		The notation $i$ denotes the imaginary unit.
		The subscript `o' means open, and `c' means closed.
		The case in which $\theta_2=0$ is allowed only for $M_{\text{fan,cc}}(\theta_1,\theta_2)$.
	\end{enumerate}

	\begin{figure}[h]
			\includegraphics[keepaspectratio, scale=0.12]{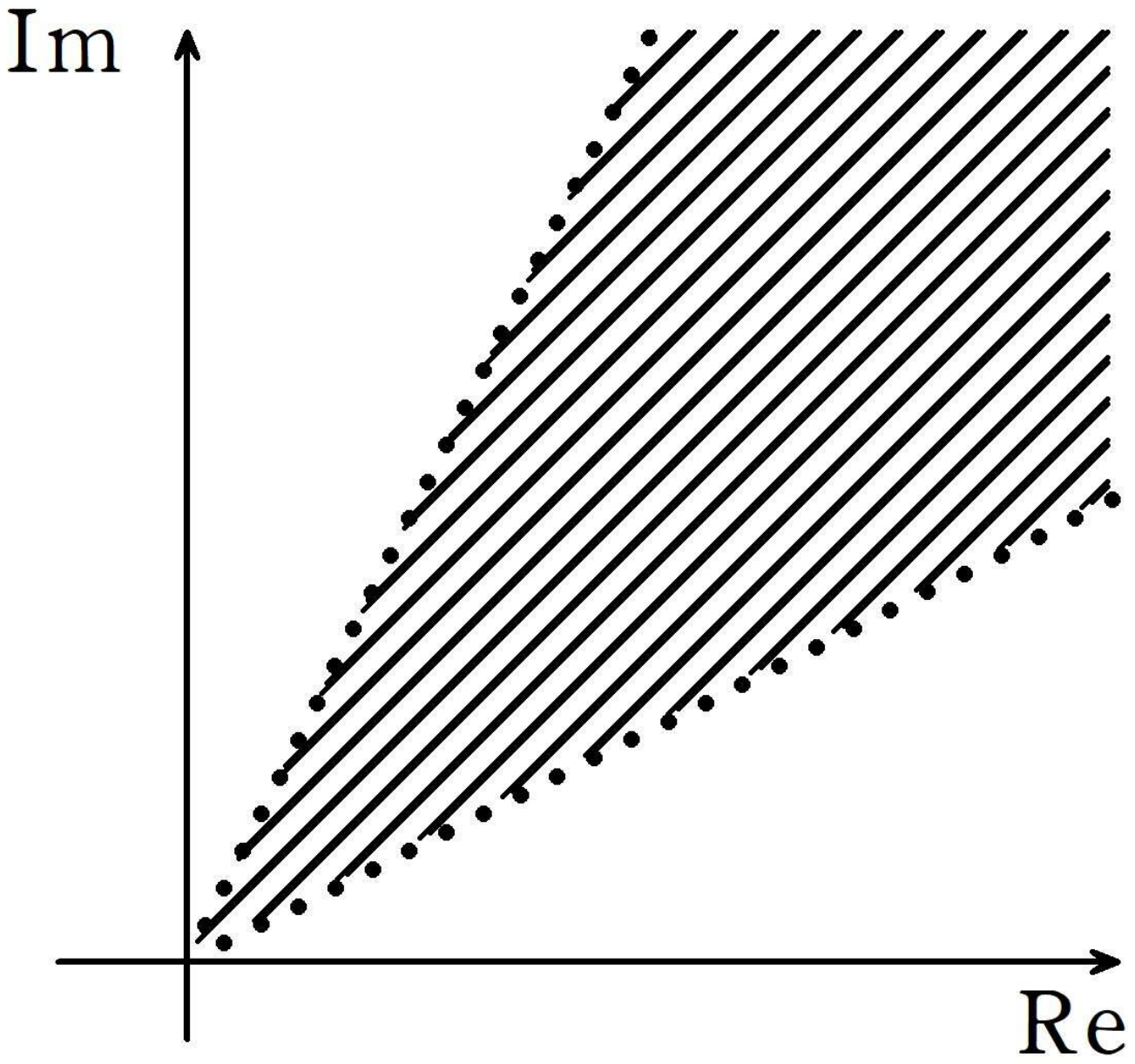}\!\!\!\!\!\!\!\!\!\!\!\!
			\includegraphics[keepaspectratio, scale=0.12]{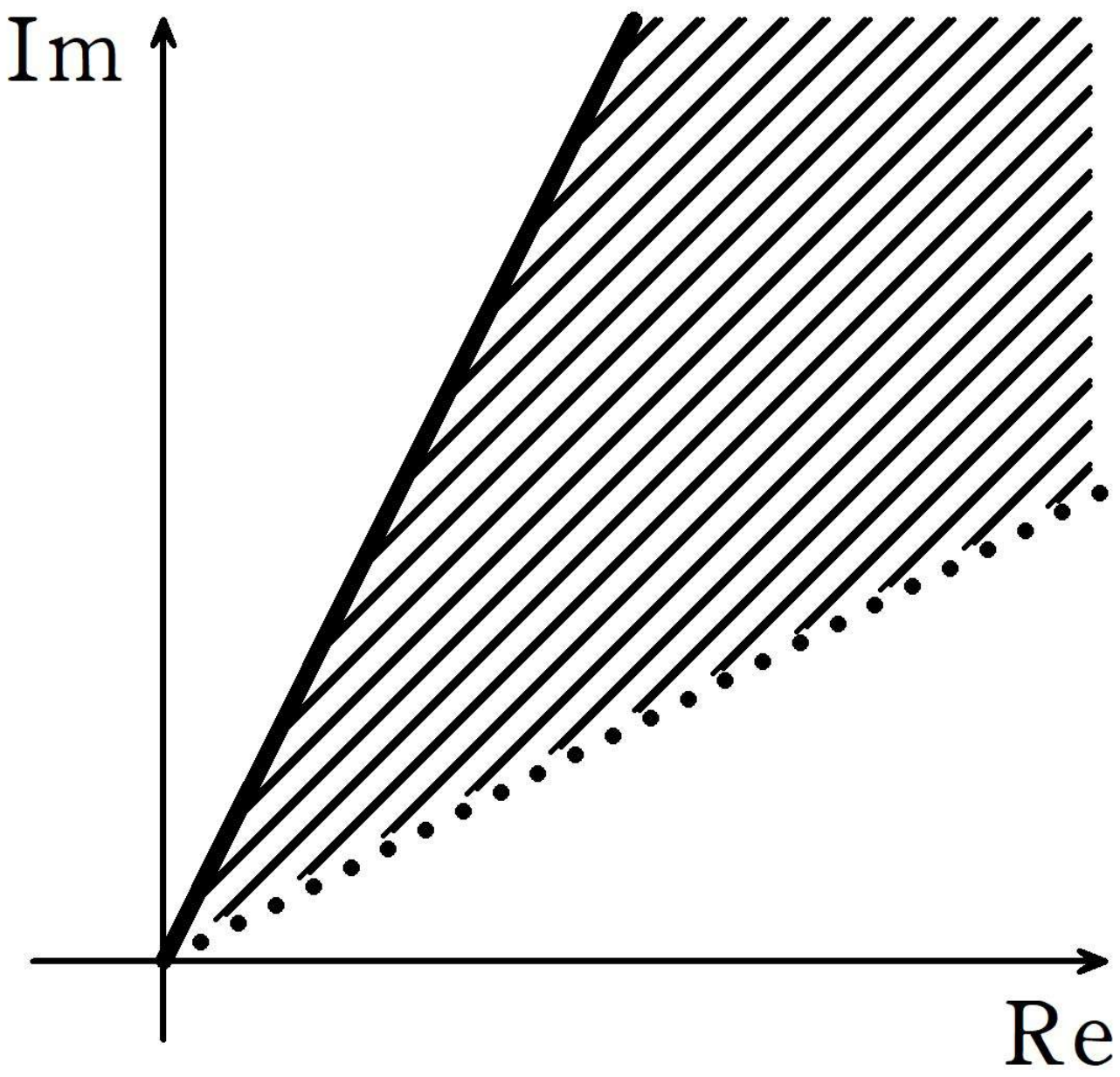}\\
			\includegraphics[keepaspectratio, scale=0.12]{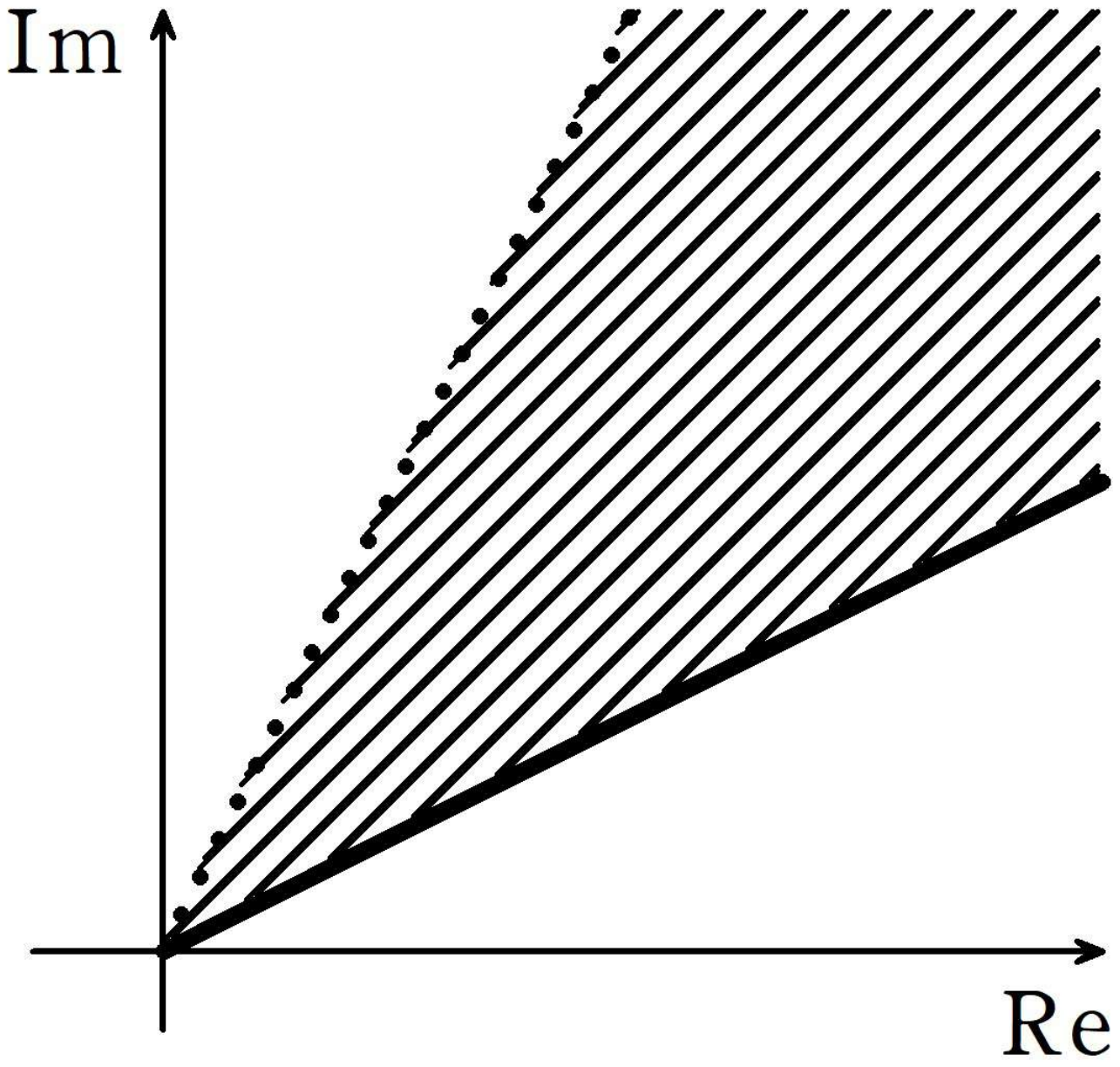}\!\!\!\!\!\!\!\!\!\!\!\!
			\includegraphics[keepaspectratio, scale=0.12]{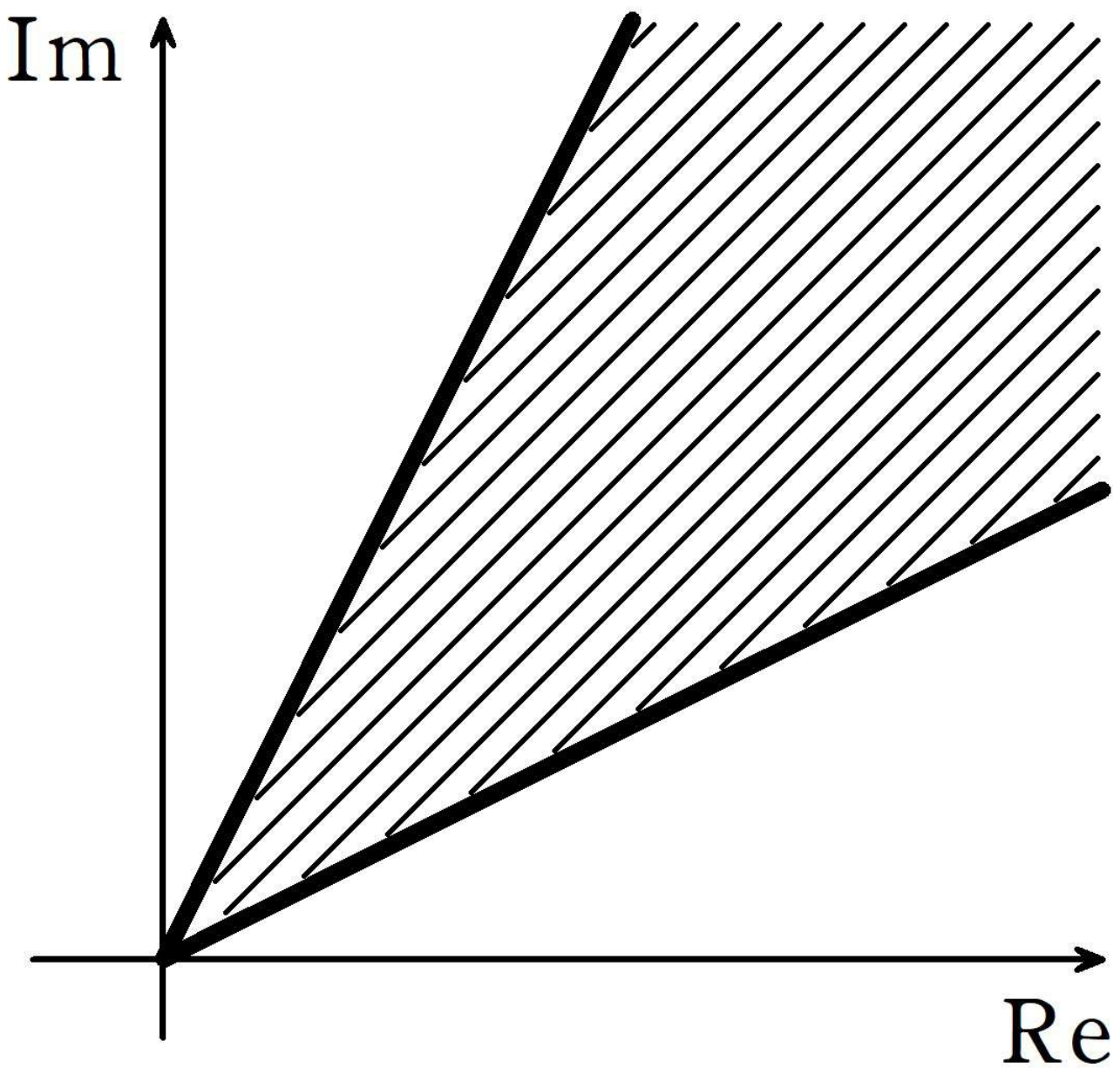}
			\caption{
					$M_{\text{fan, oo}}(\theta_1,\theta_2)$,
					$M_{\text{fan, oc}}(\theta_1,\theta_2)$,
					$M_{\text{fan, co}}(\theta_1,\theta_2)$ and
					$M_{\text{fan, cc}}(\theta_1,\theta_2)$.
				}
		\end{figure}
	
	Let $\mathcal M$ be a nonzero quasi-quadratic $A$-module in $A$.
	Set $m=\min\myval(\mathcal M)$.
	Since $\mycl_{\mathbb C}(M)=\mathbb C$ for any nonzero quasi-quadratic $\mathbb R$-module $M$ in $\mathbb C$, by Theorem \ref{thm:decomp}, the quasi-quadratic $A$-module $\mathcal M$ in $A$ is one of the following forms:
		\begin{enumerate}
				\item[(a)] When $M_m(\mathcal M)$ is either $\mathbb C$, a line (in the case (ii)) or $M_{\text{fan,cc}}(\theta,\pi)$ for some $\theta \in \mathbb R$, we have $M_n(\mathcal M)=\mathbb C$ for all $n>m$ and
				\begin{align*}
						&\mathcal M=\{x \in A\;\vert\; (\myval(x)=m \text{ and }
						\mylcp(x) \in M_m(\mathcal M))  \text{ or } \myval(x)>m\}.
					\end{align*}
				\item[(b)] When $M_m(\mathcal M)$ is one of the forms in (iii) which is not treated in (a), $M_{m+1}(\mathcal M)$ is possibly any one of (i) through (iii), $M_n(\mathcal M)=\mathbb C$ for all $n>m+1$ and
				\begin{align*}
						&\mathcal M=\{x \in A\;\vert\;
						 (\myval(x)=m \text{ and } \mylcp(x) \in M_m(\mathcal M)) \\
						&  \qquad\text{ or } (\myval(x)=m+1 \text{ and } \mylcp(x) \in M_{m+1}(\mathcal M)) \\
						&  \qquad\text{ or } \myval(x)>m+1\}.
					\end{align*}
			\end{enumerate}
	
%
\end{example}

\subsection{Finitely generated quasi-quadratic modules of the pseudo-valuation ring.}
\label{subsec:finite-qq-mod}

In this subsection, we only consider the pseudo-valuation ring
$$A=\left\{\left.\sum_{l=0}^\infty a_lx^l \in \mathbb C[\![X]\!]\;\right\vert\; a_0 \in \mathbb R\right\}.$$
We give a necessary and sufficient condition for a quasi-quadratic module of the pseudo-valuation ring $A$
to be finitely generated as a quasi-quadratic $A$-module.

\begin{lemma}\label{fg_val_ring}
	Let $A$ be as above. Then the valuation ring $\mathbb{C}[\![X]\!]$ of $\mathbb{C}(\!(X)\!)$ is
	finitely generated  as a quasi-quadratic $A$-module.
\end{lemma}
\begin{proof}
	We demonstrate that 
	$\mathbb{C}[\![X]\!]=(\pm 1)A^2+(\pm i)A^2+(\pm X)A^2+(\pm iX)A^2$. 
	It immediately follows that the right-hand side of the equality is included the left-hand side.
	We demonstrate the opposite inclusion. Take a nonzero element $f=\sum_{k=0}^{\infty}c_kX^k\in \mathbb{C}[\![X]\!]$.
	
	When $\myval(f)=2m$ for some $m\geq 0$, there exists a nonzero element $g \in A$ such that $f=c_{2m}g^2$.
	We can take real numbers $\alpha$, $\beta\in\mathbb{R}$ with $c_{2m}=\alpha+i\beta$.
	Hence we have $f=\alpha g^2+i(\beta g^2)$. Since $\alpha, \beta\in \mathbb{R}^2\cup(-\mathbb{R}^2)$,
	the element $f$ is included in the right-hand side of the equality.
	
	When $\myval(f)=2m+1$ for some $m\geq 0$, there exists a nonzero element $h\in A$ such that 
	$f=c_{2m+1}X^{2m+1}h^2$.
	We can take real numbers $\gamma$, $\delta\in\mathbb{R}$ with $c_{2m+1}=\gamma+i\delta$.
	Hence we have $f=\gamma X (X^{m}h)^2+i\bigl(\delta X(X^{m}h)^2\bigl)$. Since $\gamma, \delta\in \mathbb{R}^2\cup(-\mathbb{R}^2)$, the element $f$ is included in the right-hand side of the equality.
\end{proof}

\begin{lemma}\label{conv-fans}
	Let $A$ be as above. With the same notation as in Example \ref{ex:ex2}, the following statements hold true:
	\begin{enumerate}[(1)]
		\item Convex fans $M_{\text{\rm fan,oo}}(\theta_1, \theta_2)$, $M_{\text{\rm fan,oc}}(\theta_1, \theta_2)$ and 
		$M_{\text{\rm fan,co}}(\theta_1, \theta_2)$ for any $\theta_1, \theta_2$ with $0< \theta_2\leq \pi$ are not 
		finitely generated as quasi-quadratic $\mathbb{R}$-modules in $\mathbb{C}$. 
		\item A line $M_{\text{\rm line}}(z)$ for any $z\in\mathbb{C}$ with $z \neq 0$ 
		and a convex fan $M_{\text{\rm fan,cc}}(\theta_1, \theta_2)$ for any $\theta_1, \theta_2$ with $0\leq \theta_2\leq \pi$
		are finitely generated as a
		quasi-quadratic $\mathbb{R}$-modules in $\mathbb{C}$. 
	\end{enumerate}
\end{lemma}

\begin{proof}
	(1) We only demonstrate the case in which $M_{\text{\rm fan,oo}}(\theta_1, \theta_2)$.
	The other cases follow in the same way.
	
	Since $M_{\text{\rm fan,oo}}(\theta_1, \theta_2)$ is a convex ray, the following
	inequality holds true:
	$\arg (z+w)\leq \max \{\arg z, \arg w\}$ for any $z,w\in M_{\text{\rm fan,oo}}(\theta_1, \theta_2)$,
	where the notation $\arg v$ denotes the argument for $v\in\mathbb{C}$.
	Assume for contradiction that $M_{\text{\rm fan,oo}}(\theta_1, \theta_2)$ is a finitely generated 
	quasi-quadratic $\mathbb{R}$-module. There exist elements $z_1,\ldots,z_n \in 
	M_{\text{\rm fan,oo}}(\theta_1, \theta_2)$ such that
	$
	M_{\text{\rm fan,oo}}(\theta_1, \theta_2)=\mathbb{R}^2z_1+\cdots+\mathbb{R}^2z_n
	$. It follows that $\arg z_l=\max\{\arg z_1, \ldots, \arg z_n\}$ for some $1\leq l\leq n$.
	We can take an element $\theta_0\in\mathbb{R}$ with $\arg z_l< \theta_0<\theta_1+\theta_2$.
	Since $e^{i\theta_0}\in M_{\text{\rm fan,oo}}(\theta_1, \theta_2)$, we have
	$
	e^{i\theta_0}=r_1^2z_1+\cdots+r_n^2 z_n 
	$ for some $r_1,\ldots,r_n\in\mathbb{R}$.
	Hence it follows that
	$
	\theta_0=\arg(r_1^2z_1+\cdots+r_n^2 z_n)\leq \max_{1\leq i\leq n}\{\arg z_i\}\leq \arg z_l
	$. Contradiction.

	(2) We obviously have $M_{\text{\rm line}}(z)=\mathbb R^2 z+\mathbb R^2(-z)$.
	This means that $M_{\text{\rm line}}(z)$ is finitely generated.
	We next show that a convex fan $M_{\text{\rm fan,cc}}(\theta_1, \theta_2)$ is finitely generated.
	Set $z_1=e^{i\theta_1}$, $z_2=e^{i(\theta_1+\theta_2/2)}$ and $z_3=e^{i(\theta_1+\theta_2)}$.
	Then we have the following equalities:
		\begin{align*}
			M_{\text{\rm fan,cc}}(\theta_1, \theta_2)
			&=\{sz_1+tz_2\;\vert\;s,t\in\mathbb{R}^2\}
			\cup\{uz_2+vz_3\;\vert\;u,v\in\mathbb{R}^2\}\\
			&=z_1\mathbb{R}^2+z_2\mathbb{R}^2+z_3\mathbb{R}^2.
			\end{align*}
\end{proof}

We are now ready to demonstrate our main theorem in this section.

\begin{theorem}\label{fg_thm}
	Let $\mathcal M$ be a quasi-quadratic $A$-module in $A$. Set $m=\min(\myval(\mathcal M))$.
	The quasi-quadratic module $\mathcal M$ is finitely generated as a quasi-quadratic $A$-module in $A$
	if and only if the quasi-quadratic modules $M_m(\mathcal M)$ and $M_{m+1}(\mathcal M)$ are finitely generated as 
	quasi-quadratic $\mathbb{R}$-modules in $\mathbb{C}$
	if and only if one of the following conditions is satisfied:
	\begin{itemize}
		\item $M_m(\mathcal M)$ is either $\mathbb C$, a line or a  convex fan of the form $M_{\text{\rm fan,cc}}(\theta, \pi)$ for some $\theta \in \mathbb R$.
		\item $M_m(\mathcal M)=M_{\text{\rm fan,cc}}(\theta_1, \theta_2)$ for some $\theta_1 \in \mathbb R$ and $0<\theta_2<\pi$ and $M_{m+1}(\mathcal M)$ is either $\mathbb C$, a line or a convex fan of the form $M_{\text{\rm fan,cc}}(\theta'_1, \theta'_2)$ for some $\theta'_1 \in \mathbb R$ and $0 \leq \theta'_2 \leq \pi$.
	\end{itemize}
\end{theorem}
\begin{proof}
	The second `if and only if' part follows from Example \ref{ex:ex2} and Lemma \ref{conv-fans} once the first `if and only if' part is proved.
	We prove the first `if and only if' part in the rest of the proof.
	
	We first assume that $\mathcal M$ is a finitely generated quasi-quadratic $A$-module in $A$.
	There is a finite number of nonzero elements $g_1,\ldots,g_n\in \mathcal M$ with
	$\mathcal M=A^2g_1+\cdots+A^2g_n$. Set $S_k:=\{i\;\vert\; 1\leq i\leq n, \myval(g_i)=k\}$ for $k \geq m$.
	We first prove that $M_m(\mathcal M)$ is finitely generated.
	Note that $S_m$ is not empty because $\min(\myval(\mathcal M))=m$.
	It is immediate that $\mylcp(g_i)\in M_m(\mathcal M)$ for any $i\in S_m$.
	By Proposition \ref{prop:M}, we have $\sum_{i\in S_m}\mathbb{R}^2\mylcp(g_i)\subset M_m(\mathcal M)$.
	We next prove the opposite inclusion. Take a nonzero element $c\in M_m(\mathcal M)$.
	We can take a nonzero element $f\in \mathcal M$ with $c=\mylcp(f)$ and $\myval(f)=m$.
	For each $1\leq i\leq m$, there are elements $a_i(X)\in A$ such that 
	$f=a_1(X)^2g_1+\cdots+a_n(X)^2 g_n$ by the assumption.
	Then we have 
	\begin{align*}
	c&=\mylcp(f)
	=\sum_{i\in S_m}a_i(0)^2\mylcp(g_i)\in \sum_{i\in S_m}\mathbb{R}^2\mylcp(g_i).
	\end{align*}
	It implies that $M_m(\mathcal M)$ is finitely generated.
	
	We next show that $M_{m+1}(\mathcal M)$ is finitely generated.
	When $\mathcal M$ is one of the $A$-quasi-quadratic modules given in (a) of Example \ref{ex:ex2}, we have $M_{m+1}(\mathcal M)=\mathbb C$.
	It is finitely generated because $\mathbb C = \mathbb R^2 + (-\mathbb R^2) + (i \mathbb R^2)+(-i\mathbb R^2)$.
	
	We concentrate on the case in which $\mathcal M$ is one of the quasi-quadratic $A$-modules given in (b) of Example \ref{ex:ex2}.
	Let $f$ be an arbitrary element of $\mathcal M$ with $\myval(f)=m+1$.
	There are $a_1(X), \ldots, a_n(X)\in A$ such that 
	$f=a_1(X)^2g_1+\cdots+a_n(X)^2 g_n$ by the assumption.
	Since $\myval(f)=m+1$, we have $\sum_{i \in S_m}a_i(0)^2\mylcp(g_i)=0$.
	If $a_k(0) \neq 0$ for some $k \in S_m$, we have 
	$0 \neq a_k(0)^2\mylcp(g_k) \in M_m(\mathcal M)$ and 
	\begin{align*}
	&-a_k(0)^2\mylcp(g_k)
	=\sum_{i \in S_m, i \neq k}a_i(0)^2\mylcp(g_i) \in M_m(\mathcal M).
	\end{align*}
	It is a contradiction because there are no elements $u \neq 0$ in $M_m(\mathcal M)$ with $- u \in M_m(\mathcal M)$ in this case.
	Therefore, we have $a_i(0)=0$ for every $i \in S_m$.
	It implies that $\myval(a_i(X)^2g_i)>m$ for every $1 \leq i \leq n$. 
	Using these inequalities, we finally get the equality $$\mylcp(f)=\sum_{i \in S_{m+1}}a_i(0)^2\mylcp(g_i).$$
	It means that $M_{m+1}(\mathcal M)$ is generated by $\{\mylcp(g_i)\;\vert\; i \in S_{m+1}\}$ by the definition of $M_{m+1}(\mathcal M)$.

	We demonstrate the opposite implication. Suppose that $M_m(\mathcal M)$ and $M_{m+1}(\mathcal M)$ are finitely generated
	quasi-quadratic $\mathbb{R}$-modules in $\mathbb{C}$.
	Let $a_{i1}, \ldots a_{il_i}$ be generators of $M_{m+i}(\mathcal M)$ for $i=0,1$.
	Set $$\mathcal N_i=\{f \in A\;\vert\; \myval(f)=m+i \ \mbox{and}\ \mylcp(f)\in M_{m+i}(\mathcal M))\}$$ for $i=0,1$.
	We have $\mathcal N_i \subseteq \sum_{j=1}^{l_i}a_{ij}X^{m+i}A^2$.
	In fact, take a nonzero element $f \in \mathcal N_i$.
	We have $f=\mylcp(f)X^{m+i}g$ for some $g \in A$ with $\pi(g)=1$.
	Since any strict units in $A$ admit a square root, we have 
	\begin{align*}
	f \in \mylcp(f)X^{m+i}A^2 &\subseteq \sum_{j=1}^{l_i}  a_{ij}\mathbb R^2X^{m+i}A^2 =\sum_{j=1}^{l_i}a_{ij}X^{m+i}A^2.
	\end{align*}
	
	By Example \ref{ex:ex2}, we have 
	\begin{align*}
		&\mathcal M=\{x \in A\;\vert\; (\myval(x)=m \text{ and } \mylcp(x) \in M_m(\mathcal M))  \text{ or } \myval(x)>m\}
	\end{align*}
	or 
	\begin{align*}
		&\mathcal M=\{x \in A\;\vert\; 
		 (\myval(x)=m \text{ and } \mylcp(x) \in M_m(\mathcal M)) \\
		&  \text{ or } (\myval(x)=m+1 \text{ and } \mylcp(x) \in M_{m+1}(\mathcal M)) \\
		&  \text{ or } \myval(x)>m+1\}.
	\end{align*}
	Even in the former case, we may assume that $\mathcal M$ is of the latter form because $M_{m+1}(\mathcal M)=\mathbb C$ in the former case.
	Therefore, we only treat the case in which $\mathcal M$ is of the latter form.
	It is obvious that $\mathcal N_i \subseteq \mathcal M$ for $i=0,1$.
	These inclusions imply $\sum_{j=1}^{l_i}a_{ij}X^{m+i}A^2 \subseteq \mathcal M$ for $i=0,1$ because $a_{ij}X^{m+i} \in \mathcal N_i$ and $\mathcal M$ is a quasi-quadratic $A$-module.
	We finally have 
	\begin{align*}
	\mathcal M &= \mathcal N_0 \cup \mathcal N_1 \cup X^{m+2}\mathbb C[\![X]\!]\\ &\subseteq \sum_{i=0}^1 \sum_{j=1}^{l_i}a_{ij}X^{m+i}A^2 + X^{m+2}\mathbb C[\![X]\!]  \\
	&\subseteq \mathcal M.
\end{align*}
	It implies that $$\mathcal M = \sum_{i=0}^1 \sum_{j=1}^{l_i}a_{ij}X^{m+i}A^2 + X^{m+2}\mathbb C[\![X]\!].$$
	This equality implies that $\mathcal M$ is finitely generated by Lemma \ref{fg_val_ring}.
\end{proof}

\section{Appendix: When $F_0$ is of characteristic two}\label{sec:char2}
We consider the case in which the residue field $F_0$ is of characteristic two.
From now on, let $A$ be a pseudo-valuation ring whose residue field is of characteristic two such that any strict units admit a square root.
We use the following lemma:
\begin{lemma}[Basic lemma]\label{lem:char_two1}
	Any element in $\mathfrak m$ is the sum of squares of two elements in $A$.
\end{lemma}
\begin{proof}
	Let $x \in \mathfrak m$.
	Since $\pi_0(1+x)=1$ and $\pi_0(-1)=-1=1$, both $1+x$ and $-1$ are strict units in $A$.
	There exist $u,v \in A$ with $u^2=1+x$ and $v^2=-1$ because strict units admit a square root in $A$.
	We have $x=u^2+v^2$.
\end{proof}

We give a corollary of the basic lemma.

\begin{corollary}\label{cor:char_two1}
	Let $\mathcal M$ be a quasi-quadratic $A$-module in $A$.
	Let $x \in \mathcal M$ and $y \in A$.
	If $\myval(x)<\myval(y)$, then $y \in \mathcal M$.
\end{corollary}
\begin{proof}
	By the assumption, we have $\myval(y/x)>e$ and $y/x \in \mathfrak m$.
	There exist $u,v \in A$ with $y/x=u^2+v^2$.
	It implies that $y=(u^2+v^2)x \in \mathcal M$.
\end{proof}

For presenting structure theorems of quasi-quadratic modules, we introduce several notations.

\begin{definition}\label{def:cut}
	A subset $S$ of $G_{\geq e}$ is called a \textit{final segment} if $y \in S$ whenever $y>x$ for some $x \in S$.
	Let $\mathfrak S$ denotes the set of final segments and $\mathfrak S_{\min}$ denote the set of final segments having smallest elements.
	For any final segment $S$, we set $$\Delta_1(S)=\{0\} \cup\{0 \neq x \in A\;\vert\; \myval(x) \in S \}.$$ 
	When $S$ has a smallest element, for any nonzero quasi-quadratic $F_0$-module $M$ in $F$, we set \begin{align*}
		\Delta_2(S,M)&=\{0\}\cup\{0 \neq x \in A\;\vert\; \myval(x) >\min S \text{ or }\\
		&\qquad (\myval(x)=\min S \text{ and } \mylcp(x) \in M)\}.
	\end{align*} 
	Here, $\min S$ denotes the unique smallest element of $S$.
	Consider the case in which $S$ has a smallest element.
	Note that $\Delta_1(S)=\Delta_2(S,F_0)$ when $\min S= e$ and $\Delta_1(S)=\Delta_2(S,F)$ otherwise.
\end{definition}

The following lemma asserts that the sets $\Delta_1(S)$ and $\Delta_2(S,M)$ are quasi-quadratic $A$-modules in $A$.

\begin{lemma}\label{lem:char_two_qq}
	The following assertions hold true:
	\begin{enumerate}
		\item[(1)] The set $\Delta_1(S)$ is a quasi-quadratic module in $A$ for any $S\in\mathfrak S$.
		\item[(2)] The set $\Delta_2(S,M)$ is a quasi-quadratic module in $A$ for any $S\in\mathfrak
		S_{\min}$ and for any nonzero quasi-quadratic $F_0$-module $M$ in $F$.
	\end{enumerate}
\end{lemma}
\begin{proof}
	(1) It is a routine to demonstrate that $\Delta_1(S)$ is closed
	under multiplication by the squares of elements in $A$.
	We have only to demonstrate that $\Delta_1(S)$ is closed under addition.
	Take a nonzero elements $x,y\in\Delta_1(S)$ with $x+y\neq 0$.
	It immediately follows that $\myval(x+y)\in\Delta_1(S)$ since 
	$\myval(x+y)\geq\min(\myval(x),\myval(y))\in S$ and $S$ is a final segment.

	(2) We first demonstrate the closedness under addition. Take nonzero elements $x,y\in\Delta_2(S,M)$ with $x+y\neq 0$.
	We may assume that $\myval(x)\geq \myval(y)$. We first consider the case in which $\myval(x)=\myval(y)$.
	When $\mylcp(x)+\mylcp(y)\neq 0$, it follows from Definition \ref{def:pseudo}(4)
	that $\myval(x+y)=\myval(x)$ and $\mylcp(x+y)=\mylcp(x)+\mylcp(y)$.
	This implies that $x+y\in\Delta_2(S,M)$.
	When $\mylcp(x)+\mylcp(y)= 0$, we have
	that $\myval(x+y)>\myval(x)\geq \min S$ by using Lemma \ref{lem:lcp_basic2}.
	This means that $x+y\in\Delta_2(S,M)$.
	We next consider the remaining case in which $\myval(x)>\myval(y)$.
	We get $\myval(x+y)=\myval(y)$ and $\mylcp(x+y)=\mylcp(y)$ from Definition \ref{def:pseudo}(4).
	Hence it follows that $x+y\in\Delta_2(S,M)$.
	
	We demonstrate that $\Delta_2(S,M)$ is closed under multiplication by the squares of elements in $A$.
	Take nonzero elements $x\in\Delta_2(S,M)$ and $a\in A$. When $\myval(x)>\min S$, 
	it is obvious that $\myval(a^2x)\geq\myval(x)>\min S$. Thus we have $a^2x\in\Delta_2(S,M)$.
	We next consider the case in which $\myval(x)=\min S$ and $\mylcp(x)\in M$.
	When $\myval(a)>e$, it follows that $\myval(a^2x)>\myval(x)=\min S$. This implies that $a^2x\in\Delta_2(S,M)$.
	When $\myval(a)=e$, we have 
	$\myval(a^2x)=\myval(x)=\min S$ and $\mylcp(a^2x)=\pi_0(a)^2\mylcp(x)\in M$ 
	from Definition \ref{def:pseudo}(2).
\end{proof}

\begin{theorem}\label{thm:str_chartwo}
	Let $\mathcal M$ be a quasi-quadratic $A$-module in $A$.
	\begin{enumerate}
		\item[(1)] $S=\myval(\mathcal M)$ is a final segment.
		\item[(2)] If $S$ does not have a smallest element, we have $\mathcal M=\Delta_1(S)$.
		Otherwise, $\mathcal M=\Delta_2(S,M_{\min S}(\mathcal M))$.
	\end{enumerate}
\end{theorem}
\begin{proof}
	The assertion (1) immediately follows from Corollary \ref{cor:char_two1}.
	
	We demonstrate the assertion (2) when $S$ does not have a smallest element.
	The inclusion $\mathcal M \subseteq \Delta_1(S)$ is obvious.
	We demonstrate the opposite inclusion.
	Take a nonzero element $x \in \Delta_1(S)$.
	We have $\myval(x) \in S$.
	Since $S$ does not have a smallest element, we can take $y \in \mathcal M$ with $\myval(y)<\myval(x)$.
	We get $x \in \mathcal M$ by Corollary \ref{cor:char_two1}.
	
	We finally consider the remaining case.
	Set $M=M_{\min S}(\mathcal M)$ for simplicity.
	The inclusion $\mathcal M \subseteq \Delta_2(S,M)$ is easy.
	We omit the proof.
	We prove the opposite inclusion.
	Take a nonzero $x \in \Delta_2(S,M)$.
	When $\myval(x)>\min S$, we can prove $x \in \mathcal M$ similarly in the previous case.
	When $\myval(x)=\min S$, we have $\mylcp(x) \in M$.
	By the definition of $M$, there exists $y \in \mathcal M$ such that $\myval(y)=\myval(x)$ and $\mylcp(y)=\mylcp(x)$.
	We have $x \in \mathcal M$ by Lemma \ref{lem:basic4}. 
\end{proof}

\begin{proposition}\label{prop:char_two2}
	The following assertions hold true:
	\begin{enumerate}
		\item[(1)] Let $S_1$ and $S_2$ be final segments in $\myval(A)$.
		Then, at least one of the inclusions $S_1 \subseteq S_2$ and $S_2 \subseteq S_1$ holds true.
		\item[(2)] Let $S_1$ and $S_2$ be final segments in $\myval(A)$.
		Let $M_1$ and $M_2$ be nonzero quasi-quadratic $F_0$-modules in $F$.
				\begin{enumerate}
						\item[(a)] We have $\Delta_1(S_1) \cap \Delta_1(S_2)= \Delta_1(S_1)$ and 
						$\Delta_1(S_1)+\Delta_1(S_2)=\Delta_1(S_2)$ when $S_1 \subseteq S_2$;
						\item[(b)] When $S_1, S_2 \in \mathfrak S_{\min}$, we have
						\begin{align*}
								&\Delta_2(S_1,M_1) \cap \Delta_2(S_2,M_2) \\
								&= \left\{
								\begin{array}{ll}
										\Delta_2(S_1,M_1) & \text{ if } S_1 \subsetneq S_2,\\
										\Delta_2(S_1, M_1 \cap M_2) & \text{ if } S_1=S_2 \text{ and } \\
										& \quad M_1 \cap M_2 \neq \{0\},\\
										\Delta_1(S') & \text{ if } S_1=S_2 \text{ and }\\
										& \quad M_1 \cap M_2 = \{0\},
									\end{array}
								\right.
							\end{align*}
						where $g_{\min}$ is a smallest element of $S_1$ and $S'=\{g \in S_1 \mid g > g_{\min}\}$. 
						We also have 
						\begin{align*}
								&\Delta_2(S_1,M_1) + \Delta_2(S_2,M_2) \\
								&\quad = \left\{
								\begin{array}{ll}
										\Delta_2(S_2,M_2) & \text{ if } S_1 \subsetneq S_2,\\
										\Delta_2(S_1, M_1 + M_2) & \text{ if } S_1=S_2;
									\end{array}
								\right.
							\end{align*}
						\item[(c)] When $S_1 \in \mathfrak S_{\min}$, we have
						\begin{align*}
								&\Delta_2(S_1,M_1) \cap \Delta_1(S_2) \\
								&\quad = \left\{
								\begin{array}{ll}
										\Delta_2(S_1,M_1) & \text{ if } S_1 \subseteq S_2,\\
										\Delta_1(S_2) & \text{ if } S_2 \subsetneq S_1,\\
									\end{array}
								\right.
							\end{align*}
						and $\Delta_2(S_1,M_1) + \Delta_1(S_2)$
						\begin{align*}
								& \quad = \left\{
								\begin{array}{ll}
										\Delta_1(S_2) & \text{ if } S_1 \subseteq S_2,\\
										\Delta_2(S_1,M_1) & \text{ if } S_2 \subsetneq S_1.\\
									\end{array}
								\right.
							\end{align*}
					\end{enumerate}
	\end{enumerate}
\end{proposition}
\begin{proof}
	The assertion (1) is obvious from the definition.
	Our next task is to prove assertion (2).
	The proof is similar to that of \cite[Proposition 5.10(2)]{FK}, but we give it for readers' convenience.
	When $S_1 \subseteq S_2$, we have $\Delta_1(S_1) \subseteq \Delta_1(S_2)$.
	The assertion (a) is obvious from this inclusion.
	
	We investigate the intersection and the sum of $\Delta_2(S_1,M_1)$ and $\Delta_2(S_2,M_2)$ discussed in assertion (b).
	When $S_1 \subsetneq S_2$, $\min S_2$ is smaller than any element in $S_1$ by the definition of final segments.
	Hence we obviously have $\Delta_2(S_1,M_1) \subseteq \Delta_2(S_2,M_2)$. 
	The equalities in the assertion (b) are obvious from this inclusion in this case.
	We next consider the case in which $S_1=S_2$.
	Set $g_{\min}=\min S_1$.
	The equalities on the intersection $\Delta_2(S_1,M_1) \cap \Delta_2(S_2,M_2)$ are not hard to derive.
	We omit the details.
	For the sum $\Delta_2(S_1,M_1)+\Delta_2(S_2,M_2)$, we first demonstrate the inclusion $\Delta_2(S_1,M_1)+\Delta_2(S_2,M_2) \subseteq \Delta_2(S_1,M_1+M_2)$.
	Take arbitrary elements $x_i \in \Delta_2(S_i,M_i)$ for $i=1,2$.
	We want to demonstrate $x_1+x_2 \in \Delta_2(S_1,M_1+M_2)$.
	It is obvious when at least one of $x_i$ is zero.
	It is also true when $\myval(x_1)\neq \myval(x_2)$.
	We next consider the case in which $\myval(x_1)=\myval(x_2) > g_{\min}$.
	It follows that 
	$\myval(x_1+x_2)\geq\min\{\myval(x_1),\myval(x_2)\}>g_{\min}$. Thus we get $x_1+x_2 \in \Delta_2(S_1,M_1+M_2)$.
	We next consider the case in which $\myval(x_1)=\myval(x_2)=g_{\min}$ and $\mylcp(x_1)+\mylcp(x_2) \neq 0$.
	It immediately follows from Definition \ref{def:pseudo}(4).
	In the remaining case, we have $\myval(x_1)=\myval(x_2)=g_{\min}$ and $\mylcp(x_1)+\mylcp(x_2) = 0$.
	We have $\myval(x_1+x_2)>g_{\min}$ by Lemma \ref{lem:lcp_basic2} in this case.
	We also get $x_1+x_2 \in \Delta_2(S_1,M_1+M_2)$.
	
	We next prove the opposite inclusion $\Delta_2(S_1,M_1+M_2) \subseteq \Delta_2(S_1,M_1)+\Delta_2(S_2,M_2)$.
	Take $x \in \Delta_2(S_1,M_1+M_2)$.
	When $\myval(x) > g_{\min}$, we obviously have $x \in \Delta_2(S_1,M_1)$.
	When $\myval(x)=g_{\min}$ and $\mylcp(x) \in M_i$ for some $i=1,2$, we obviously have $x \in \Delta_2(S_i,M_i)$.
	The final case is the case in which $\myval(x)=g_{\min}$ and $\mylcp(x) \not\in M_i$ for $i=1,2$.
	We can take nonzero $b_i \in M_i$ with $b_1+b_2 = \mylcp(x)$ in this case.
	We can also take $x_1 \in A$ such that $\myval(x_1) = g_{\min}$ and $\mylcp(x_1)=b_1$ by Definition \ref{def:pseudo}(3).
	The element $x_1$ belongs to $\Delta_2(S_1,M_1)$.
	Set $x_2 = x - x_1\neq 0$.
	We have $\myval(x_2)=g_{\min}$ and $\mylcp(x_2)=b_2$ by Definition \ref{def:pseudo}(4).
	It means that $x_2 \in \Delta_2(S_2,M_2)$.
	We have proven the inclusion $\Delta_2(S_1,M_1+M_2) \subseteq \Delta_2(S_1,M_1)+\Delta_2(S_2,M_2)$ and finished the proof of the assertion (b).
	
	The assertion (c) is easily seen because we have $\Delta_2(S_1, M_1) \subseteq \Delta_1(S_2)$ when $S_1 \subseteq S_2$ and the opposite inclusion holds true otherwise
	by the definition of final segments, $\Delta_2(S_1, M_1)$ and $\Delta_1(S_2)$.
\end{proof}

Recall that $\mathfrak X_{R}^N$
denotes the set of all the quasi-quadratic $R$-modules in an $R$-module $N$.
We set as follows:
\begin{align*}
	\mathcal O_{F_0,F}
	&=
	(\mathfrak{S}\setminus\mathfrak{S}_{\min})
	\sqcup \Bigl((\mathfrak{S}_{\min} \setminus \{G_{\geq e}\}) \times (\mathfrak{X}_{F_0}^F\setminus\{0\})\Bigl)
	\sqcup \Bigl(\{G_{\geq e}\} \times (\mathfrak{X}_{F_0}^{F_0}\setminus\{0\})\Bigl).
\end{align*}

\begin{theorem}\label{thm:ring4-ch2}
	The map $\Phi:\mathfrak{X}_A \to \mathcal O_{F,F_0}$ given by
	\begin{align*}
		&\Phi(\mathcal M) 
		= \left\{
		\begin{array}{ll}
			\myval(\mathcal M) & \text{ if } \myval(\mathcal M)\not\in\mathfrak{S}_{\min} ,\\
			(\myval(\mathcal M), M_{g_{\min}}(\mathcal M)) & \text{ if } \myval(\mathcal M)\in\mathfrak{S}_{\min}\\
		\end{array}
		\right.
	\end{align*}
	is a bijection, where $g_{\min}$ denotes a smallest element in $\myval(\mathcal M) $.
\end{theorem}
\begin{proof}
	The proof is very similar to that of \cite[Theorem 5.11(2)]{FK}.
	We give it for readers' convenience.
	
	The map $\Phi$ is well-defined by Proposition \ref{prop:M} and the inclusion $M_e(\mathcal M) \subseteq \pi(\mathcal M) \subseteq \pi(A) = F_0$.
	By using Lemma \ref{lem:char_two_qq},
	we can define the map $\Psi:\mathcal O_{F_0,F}\rightarrow \mathfrak{X}_A$ by sending
	$S$ to $\Delta_1(S)$ when $S\in\mathfrak{S}\setminus\mathfrak{S}_{\min}$ and by sending
	$(S,M)$ to $\Delta_2(S, M)$ when $(S,M)\in\Bigl((\mathfrak{S}_{\min} \setminus \{G_{\geq e}\}) \times 
	(\mathfrak{X}_{F_0}^F\setminus\{0\})\Bigl)\sqcup \Bigl(\{G_{\geq e}\} \times (\mathfrak{X}_{F_0}^{F_0}\setminus\{0\})\Bigl)$.
	
	For any $S\in\mathfrak{S}\setminus\mathfrak{S}_{\min}$, we obviously have $\Phi(\Psi(S))=S$.
	For any $(S,M)\in\Bigl((\mathfrak{S}_{\min} \setminus \{G_{\geq e}\}) \times (\mathfrak{X}_{F_0}^F\setminus\{0\})\Bigl)
	\sqcup \Bigl(\{G_{\geq e}\} \times (\mathfrak{X}_{F_0}^{F_0}\setminus\{0\})\Bigl)$,
	we get $\myval(\Delta_2(S,M))=S$ by Definition \ref{def:pseudo}(3).
	We next demonstrate that $M_{g_{\min}}(\Delta_2(S,M))=M$.
	Since $M_{g_{\min}}(\Delta_2(S,M))=\{\mylcp(x)\;\vert\;x\in\Delta_2(S,M)\;\text{and}\;\myval(x)=g_{\min}\}\cup\{0\}$, 
	we get $M_{g_{\min}}(\Delta_2(S,M))\subseteq M$. To show the opposite inclusion, take a nonzero element $c\in M$. By Definition 
	\ref{def:pseudo}(3), there exists an element $x\in K$ with $\mylcp(x)=c$ and $\myval(x)=g_{\min}$.
	When $g_{\min}>e$, it follows that $x\in\mathfrak{m}\subset A$. When $g_{\min}=e$ and $M\in\mathfrak{X}_{F_0}^{F_0}$,
	we have $x\in A$ from Lemma \ref{lem:basic}.
	These imply that $c\in M_{g_{\min}}(\Delta_2(S,M))$.
	Hence we have 
	\begin{align*}
	\Phi(\Psi(S,M))&=(\myval(\Delta_2(S,M)),M_{g_{\min}}(\Delta_2(S,M)))
	=(S,M).
	\end{align*}
	
	Take a quasi-quadratic module $\mathcal M\in\mathfrak{X}_A$.
	We next demonstrate that $\Psi(\Phi(\mathcal M))=\mathcal M$.
	When $\myval(\mathcal{M})\not\in\mathfrak{S}_{\min}$, 
	we have
	$$
	\Psi(\Phi(\mathcal M))=\Psi(\myval(\mathcal M))=\Delta_1(\myval(\mathcal M))=\mathcal M
	$$
	from Theorem \ref{thm:str_chartwo}(2).
	When $\myval(\mathcal M)\in\mathfrak{S}_{\min}$, it follows from Theorem \ref{thm:str_chartwo}(2)
	that
	\begin{align*}
	\Psi(\Phi(\mathcal M)) &=\Psi(\myval(\mathcal M),M_{g_{\min}}(\mathcal M))\\
	&=
	\Delta_2(\myval(\mathcal M),M_{g_{\min}}(\mathcal M))\\
	&=\mathcal M.
	\end{align*}
	We have finished to prove that $\Phi$ and $\Psi$ are the inverses of the others.
\end{proof}

\bibliography{sn-bibliography}


\end{document}